\documentclass[10pt]{amsart}

\usepackage{amsmath, amsfonts, amssymb, xypic}
\usepackage{latexsym, verbatim}
\usepackage{graphicx}
\usepackage{color}
\usepackage{url}
\usepackage[table]{xcolor}
\usepackage{hyperref}
\usepackage{amsmath}

\newtheorem{thm}{Theorem}[section]
 \newtheorem{cor}[thm]{Corollary}
 \newtheorem{lem}[thm]{Lemma}
 \newtheorem{prop}[thm]{Proposition}

\theoremstyle{definition}
 \newtheorem{ex}[thm]{Example}
 \newtheorem{rmk}[thm]{Remark}

 \newtheorem{defn}[thm]{Definition}

\newcommand{\Z}{\mathbb{Z}}
\newcommand{\R}{\mathbb{R}}
\newcommand{\CC}{\mathbb{C}}

\newcommand{\PP}{\mathbb{P}}

\newcommand{\cA}{{\mathcal A}}

\newcommand{\Hh}{{\mathcal H}}

\newcommand{\ov}{\overline}

\newcommand{\ot}{\otimes}

\newcommand{\Imt}{{\rm Im}\,}
\newcommand{\Ke}{{\rm Ker}\,}
\newcommand{\Cok}{{\rm Coker}\,}

\newcommand{\del}{\partial}
\newcommand{\delb}{{\bar \partial}}
\newcommand{\mub}{{\bar \mu}}

\newcommand{\wt}{\widetilde}
\newcommand{\Img}{\mathrm{Im}}
\newcommand{\Dol}{\mathrm{Dol}}
\newcommand{\dR}{\mathrm{dR}}

\newcommand{\Mh}{(M,J, \langle - , - \rangle)}

\newcommand{\greycell}{\cellcolor{black!8}}

\title{Dolbeault cohomology for almost complex manifolds}

\author[J. Cirici]{Joana Cirici}
\address[J. Cirici]{Department of Mathematics and Computer Science, Universitat de Barcelona\\
Gran Via 585\\
08008 Barcelona }
\email{jcirici@ub.edu}

\author[S. Wilson]{Scott O. Wilson}
  \address[S. Wilson]{Department of Mathematics, Queens College, City University of New York, 65-30 Kissena Blvd., Flushing, NY 11367}
  \email{scott.wilson@qc.cuny.edu}

\thanks{J. Cirici would like to acknowledge partial support from the AEI/FEDER, UE (MTM2016-76453-C2-2-P)
and the Serra H\'{u}nter Program}

\thanks{S. Wilson acknowledges support provided by a PSC-CUNY Award, jointly funded by The Professional Staff Congress and The City University of New York.}

\keywords{Almost complex manifolds, Dolbeault cohomology, Fr\"{o}licher spectral sequence, Hodge filtration, 
nilmanifolds, nearly K\"{a}hler manifolds, harmonic forms, Hodge theory.}
\subjclass[2010]{32Q60, 53C15}

\begin{document}

\begin{abstract}
This paper extends Dolbeault cohomology and its surrounding theory to arbitrary almost complex manifolds. We define a spectral sequence converging to ordinary cohomology, whose first page is the Dolbeault cohomology, and develop a harmonic theory which injects into Dolbeault cohomology. Lie-theoretic analogues of the theory are developed which yield important calculational tools for Lie groups and nilmanifolds. Finally, we develop applications to maximally non-integrable manifolds, including nearly K\"ahler $6$-manifolds, and show Dolbeault cohomology can be used to prohibit the existence of nearly K\"ahler metrics.
\end{abstract}

\maketitle

\setcounter{tocdepth}{1}
\tableofcontents

\section{Introduction}
Dolbeault cohomology is one of the most fundamental analytic invariants of complex manifolds, relating holomorphic complexity with topology by means of the Fr\"{o}licher spectral sequence. The different incarnations of Dolbeault cohomology through differential forms, harmonic subspaces of forms, and sheaf theory have made it a powerful and central tool in the study, classification, and deformation theory of complex structures, both in differential geometry and in algebraic geometry. 

Almost complex geometry encompasses complex, symplectic, and K\"{a}hler geometry within a quite general differential framework. It also arises naturally in theoretical physics via the consideration of general geometries in string theory and supersymmetry. An \textit{almost complex structure} on a smooth manifold $M$ is an endomorphism of the tangent bundle,
$J:TM\to TM$, squaring to minus the identity, $J^2=-I$. To such a structure, one associates a tensor
field $N_J$, called the \textit{Nijenhuis tensor}. The celebrated Newlander-Nirenberg theorem states that $N_J\equiv 0$
if and only if $J$ is actually a complex structure, thus furnishing the manifold with a holomorphic atlas.
In this case, the almost complex structure is said to be \textit{integrable}.

In real dimension four, there are many examples of almost complex manifolds that do not admit integrable structures.
In this case, obstructions are attained using a variety of tools, ranging from the
Enriques-Kodaira classification of complex surfaces and Chern number inequalities, to gauge theory
via Seiberg-Witten invariants. In real dimensions six and greater, it is not known whether there exists any almost
complex manifold not admitting an integrable structure. 
On the other hand, any symplectic manifold admits a compatible almost complex structure, which in the integrable case, defines a K\"{a}hler structure.

There are some relatively recent cohomological approaches for understanding the properties of almost complex manifolds. Motivated by Donaldson's open question
on tamed versus compatible symplectic forms \cite{Donaldson}, Li and Zhang \cite{LZ} introduced the so-called pure-and-full cohomologies, as subgroups of the de Rham cohomology of an almost complex manifold, which have been studied by many others (see for example \cite{Dra2, ATZ, HMT}). 
Another notable development is the Hodge theory of Tseng and Yau \cite{TY1, TY2}, who introduced new cohomology theories for symplectic manifolds. That theory has resurrected the classical harmonic symplectic theory introduced by Brylinski \cite{Brylinski}.
Related work has been developed in the setting of generalized complex geometry, a field introduced by Hitchin \cite{Hitchin} and further explored by Gualtieri \cite{Gualtieri}, which
contains complex and symplectic structures as extremal special cases. 
In this context, there is a naturally defined Dolbeault cohomology extending that of complex manifolds (\cite{Cavalcanti,CavalcantiDol}). However, for symplectic manifolds, this recovers ordinary de Rham cohomology.  
Additional progress has been made outside the K\"ahler world by Verbitsky's work on a harmonic theory for nearly K\"{a}hler manifolds \cite{Verbitsky2, Verbitsky}. As we will see below, that approach, deriving local identities for differential operators on nearly K\"{a}hler manifolds, is also a useful tool to study the Dolbeault cohomology defined here. It has also strongly influenced our work on harmonic symmetries for almost K\"{a}hler manifolds \cite{AK}, further developed by Tardini and Tomassini in \cite{TaToPreprint}, and influenced the complementary Hermitian case presented by the second author in \cite{SWH}.

The cohomological approaches mentioned above, while well attuned to studying various 
geometric structures, such as compatible metrics or symplectic structures, do not fulfill the natural desire for a cohomological theory of all almost complex manifolds that both captures topological features
and reduces to the known Dolbeault theory in the integrable case. 
In this paper, we introduce a natural extension of Dolbeault cohomology to all almost complex manifolds and establish some first foundational results concerning this invariant. These include a Fr\"{o}licher-type spectral sequence, an approach via harmonic theory, as well as applications to nilmanifolds, compact Lie groups, and nearly K\"{a}hler manifolds, opening up the door 
for the development of Dolbeault-based approaches to the study, classification, and deformation of almost complex manifolds.
\medskip

The extension of an almost complex structure $J:TM\to TM$ to the complexified tangent bundle of a manifold $M$ allows for a decomposition into $\pm i$-eigenspaces. This decomposition induces a bigrading on the complex de Rham algebra of differential forms
\[\cA^k_{\dR}(M)\otimes_{\R}\CC=\bigoplus_{p+q=k}\cA^{p,q}.\]
In the integrable case, the differential decomposes as $d=\delb+\del$ where $\delb$ has bidegree $(0,1)$ and $\del$ is its complex conjugate. The equation $d^2=0$ implies that $\delb^2=0$ and hence one defines Dolbeault cohomology as the cohomology with respect to the operator $\delb$.
In the general non-integrable case, there are two additional terms
\[d=\mub+\delb+\del+\mu\] arising from the Nijenhuis tensor, where $\mub$ has bidegree $(-1,2)$ and $\mu$ is its complex conjugate.
The equation $d^2=0$ decomposes into a set of seven equations, 
including $\mub^2=0$ as well as $\delb\mub+\mub\delb=0$.
In this case, $\delb^2\neq 0$, and so $\delb$-cohomology is not well-defined. However, the above two equations allow one to define the $\mub$-cohomology
vector spaces  
\[H_\mub^{p,q}(M):={\Ke(\mub:\cA^{p,q}\longrightarrow \cA^{p-1,q+2})\over \Img(\mub:\cA^{p+1,q-2}\longrightarrow \cA^{p,q})},
\]
and imply that
$\delb$ induces a morphism of bigraded vector spaces 
\[\delb:H^{p,q}_\mub(M)\to H^{p,q+1}_\mub(M).\] Furthermore, 
the equation $\mub \del + \del \mub +\delb^2=0$, which also follows from $d^2=0$, implies that $\delb$ squares to zero on $H_\mub^{*,*}(M)$.  We define the \textit{Dolbeault cohomology of $M$} by
\[H_{\Dol}^{p,q}(M):=H^{q}(H_\mub^{p,*}(M),\delb).
\]
This is a new invariant for almost complex manifolds which is functorial with respect to differentiable maps compatible with the almost complex structures. It generalizes the Dolbeault cohomology of complex manifolds, since in the integrable case, for which $\mub \equiv 0$, the cohomology $H_\mub^{*,*}(M)$ is the space of all forms. 
The notion of Hodge number then makes sense in the almost complex setting,  whenever Dolbeault cohomology is finite-dimensional.

A simple example involving a well-known character, the Kodaira-Thurston manifold, indicates the potential of this invariant. Recall this is a compact 4-dimensional manifold admitting both complex and symplectic structures, but not in a compatible way.
As we show below, any left-invariant almost complex structure on the Kodaira-Thurston manifold
has one of the following Hodge diamonds for the left-invariant Dolbeault cohomology,
where the diamond on the left corresponds to integrable structures while the one on the right corresponds to non-integrable ones. In particular, the Hodge diamond classifies integrability:
\[
\xymatrix@C=.1pc@R=.1pc{
 &&1\\
 &2&&1\\
 1&&2&&1\\
 &1&&2\\
 &&1
}\quad\quad\quad\quad\quad\quad
\xymatrix@C=.1pc@R=.1pc{
&&1\\
 &2&&1\\
 0&&4&&0
 \\
 &1&&2\\
 &&1
}
\]
A quick inspection on the above diamonds shows that, in contrast with the complex setting, 
Dolbeault cohomology is not upper semi-continuous for small deformations from integrable to non-integrable structures. This behaviour indicates we should not expect Dolbeault cohomology of almost complex manifolds to enjoy all the known good properties that are satisfied  in the integrable case. 

As another important example, 
while the Fr\"{o}licher spectral sequence of any compact complex surface degenerates at the first page, this is not the case for compact 4-dimensional almost complex manifolds. An example is given below by the filiform nilmanifold, a compact 4-dimensional manifold not admitting any integrable structure. Therefore, in the case of compact almost complex 4-manifolds, degeneration of the Fr\"{o}licher spectral sequence may be understood as a new obstruction to integrability.

To provide some further context for the developments in this paper, we recall that in 1954 Hirzebruch asked a series of fundamental questions related to the topology and geometry of smooth and complex manifolds \cite{Hirzebruch}. Among these is Problem 20, attributed to Kodaira and Spencer, which concerns Hermitian structures on almost complex manifolds, as we next  explain.
For any choice of compatible metric on an almost complex manifold, one can define the $\delb$-Laplacian  $\Delta_\delb$ using $\delb$ and its formal adjoint $\delb^*$. As noted by Hirzebruch in loc. cit., the $\delb$-Laplacian is elliptic even in the non-integrable case.
In particular, for any compact almost complex manifold with an almost Hermitian structure,
one can consider the finite-dimensional spaces $\Hh^{p,q}_\delb:=\Ke(\Delta_\delb)\cap\cA^{p,q}$ of $\delb$-harmonic forms and the numbers $h_\delb^{p,q}:=\dim \Hh_\delb^{p,q}$. 
In the integrable setting, these are just the Dolbeault numbers, and so are metric-independent. Hirzebruch writes:

\begin{quote}
Is $h_\delb^{p,q}$ independent of the choice of the
Hermitian structure? If not, give some other definition of these numbers which depend only on the almost-complex structure, and which generalizes the numbers $h_\delb^{p,q}=\dim H^{p,q}_\delb$ for a compact complex manifold.
\end{quote}

In 2013, Kotschick provided an updated account of Hirzebruch's problem list \cite{Problem}. On Problem 20, Kotschick writes ``There seems to have been no progress at all on this problem, which asks for a development of harmonic Dolbeault theory on arbitrary almost complex manifolds''.
The first part of Hirzebruch's question has been 
very recently answered negatively by Holt and Zhang in \cite{HZ}. Specifically,
they show that there is an almost complex structure on the Kodaira-Thurston manifold such that $h^{0,1}$ varies with
different choices of Hermitian metrics.
The Dolbeault cohomology theory introduced in the present paper 
settles the second part of Hirzebruch's question.

We end this introduction with an overview of the properties and applications of Dolbeault cohomology
for almost complex manifolds that are developed in this paper.

\subsection*{Fr\"{o}licher spectral sequence}
One main feature of Dolbeault cohomology is its relation to the Betti numbers via the Fr\"{o}licher spectral sequence, defined as the spectral sequence associated to the column filtration.
A key observation of the present work is that one may modify 
the classical Hodge filtration for complex manifolds by taking
into account the presence of $\mub$ and making it compatible with the total differential also in the non-integrable case.
We generalize the results of Fr\"{o}licher \cite{Fro} for complex manifolds, by showing that 
the 
Dolbeault cohomology of every almost complex manifold arises in the first stage of the spectral sequence
associated to this new Hodge filtration, which converges to the complex de Rham cohomology of the manifold (Theorem  \ref{E1HDol}):
\[H_{\Dol}^{p,q}(M)\cong E_1^{p,q}(M)\Longrightarrow H^{p+q}_{dR}(M,\CC).\]
Since the new Hodge filtration is compatible with morphisms of almost complex manifolds, 
it follows that the entire spectral sequence is functorial and a well-defined invariant of almost complex manifolds.
In the integrable case, for which $\mub\equiv 0$, we recover the usual Hodge
filtration and Fr\"{o}licher spectral sequence in addition to the usual description of Dolbeault cohomology.
Note that the $E_\infty$-page gives a natural bigrading on the complex de Rham cohomology of any almost complex manifold, generalizing the existing bigrading for complex manifolds.

The existence of a spectral sequence converging to de Rham cohomology allows for applications outside Dolbeault cohomology. For instance, in \cite{AK} we use this spectral sequence to study the numbers \[\ell^{p,q}:=\dim \Hh_d^{p,q}\]
of $d$-harmonic forms of pure bidegree $(p,q)$. These numbers enjoy very special properties in the almost K\"{a}hler case. We are able to determine them using the fact that they are bounded by any page of the spectral sequence in the corresponding bidegree.

\subsection*{Harmonic theory}
Via the choice of a Hermitian metric, one may consider the various Laplacians $\Delta_\delta$ 
associated to the operators $\delta=\mub, \delb,\del,\mub$ or $d$ as well as the corresponding spaces of \textit{$\delta$-harmonic forms}
\[\Hh_\delta^{p,q}:=\Ke(\Delta_\delta)\cap \cA^{p,q}.\]
The operators $\Delta_\delb$, $\Delta_\del$ and $\Delta_d$ are elliptic, but
$\Delta_\mub$ and $\Delta_\mu$ are not.

In order to compare Dolbeault cohomology with harmonic forms, we first consider the spaces of \textit{$\delb$-$\mub$-harmonic forms}, given by the intersections
$\Hh^{p,q}_\delb \cap \Hh^{p,q}_\mub$, which are the same as $\Ke(\Delta_\delb + \Delta_\mub)\cap \cA^{p,q}$.
These spaces are finite-dimensional and satisfy Serre duality whenever $M$ is compact. They always inject into Dolbeault cohomology and in the extremal bidegrees $q=0$ and $q=m$, we obtain isomorphisms 
\[\Hh^{*,q}_\delb \cap \Hh^{*,q}_\mub\cong H_{\Dol}^{*,q}(M).\]
However, in general this isomorphism is not true in arbitrary bidegrees.

For almost complex manifolds whose $\mub$ operator has locally constant rank (so that the rank of $\mub:\cA^{p,q}_x\to \cA^{p-1,q+2}_x$ at the fibers
is constant as a function of $x\in M$) we consider an intermediate space as a candidate for describing Dolbeault cohomology purely in terms of harmonic forms. This is defined as follows.
First, on the space of $\mub$-harmonic forms there is an operator
\[\delb_\mub:\Hh^{p,q}_\mub\longrightarrow \Hh^{p,q+1}_\mub \quad \text{given by} \quad \delb_\mub(\alpha):=\Hh_\mub(\delb \alpha),\]
where $\Hh_\mub(\alpha)$ denotes a projection of $\alpha$ into $\mub$-harmonic forms. The locally constant rank condition ensures this is a smooth form. In Theorem \ref{HDolidentifications} we prove this
is a square zero operator whose cohomology computes the Dolbeault cohomology
\[H_{\Dol}^{*,*}(M)\cong {{\Ke(\delb_\mub)}\over{\Img(\delb_\mub)}}.\]

There is also a Laplacian $\Delta_{\delb_\mub}$  defined via the formal adjoint $\delb^*_\mub$ to $\delb_\mub$
and corresponding spaces of
 \textit{$\delb_\mub$-harmonic forms} $\Hh_{\delb_\mub}^{p,q}:=\Ke(\Delta_{\delb_\mub})\cap \Hh_{\mub}^{p,q}$.
We have the following \emph{Harmonic Inclusion Theorem} (\ref{HITCR}), which generalizes the Hodge decomposition theorem for compact complex manifolds: For any compact almost Hermitian manifold $M$, there is the following containment and inclusion:
\[
\Hh_\delb \cap \Hh_{\mub} \,  \subseteq  \, \Hh_{\delb_\mub}  \, \stackrel{\star}{\hookrightarrow} \, 
{\Ke(\delb_\mub)\over\Img(\delb_\mub)}\cong
H_{\Dol}(M).
\]
The injection $(\star)$ is an isomorphism if $\Img(\delb_\mub)^\perp \cong \Cok(\delb_\mub)$. This is the case, for instance, for any left-invariant almost complex structure on a finite-dimensional real Lie algebra.

The above spaces of harmonic forms lead to finite-dimensionality of Dolbeault cohomology in various situations. For instance, when $M$ is a compact manifold of dimension $2m$, the bottom and top rows $H_{\Dol}^{p,0}(M)$ and $H_{\Dol}^{p,m}(M)$ are finite-dimensional for all $p$, and $\dim H^{0,0}_{\Dol}(M)\cong \dim H_{\Dol}^{m,m}(M)$ accounts for the connected components of the manifold.
These spaces are also useful to make explicit calculations on
spaces with locally constant rank $\mub$, such as 
nilmanifolds and nearly K\"ahler $6$-manifolds.

\subsection*{Lie algebra cohomology}
Already in the integrable case, Dolbeault cohomology can be difficult to compute in general.
A framework that is particularly useful is that of Lie algebra cohomology, which allows to compute geometric invariants
for compact Lie groups as well as for nilmanifolds. A main advantage of this framework is that 
the computation of Dolbeault cohomology is reduced, by construction, to finite-dimensional linear algebra problems.

An almost complex structure on a real Lie algebra $\mathfrak{g}$ of dimension $2m$ defines a bigrading
on the Chevalley-Eilenberg dg-algebra $\cA^*_{\mathfrak{g}_\CC}$ associated to the complexification $\mathfrak{g}_\CC$
and its differential decomposes into $d=\mub+\delb+\del+\mu$.
This gives an obvious notion of \textit{Lie algebra Dolbeault cohomology}
\[H_{\Dol}^{p,q}(\mathfrak{g},J):=H^{q}(H_\mub^{p,*}(\mathfrak{g},J),\delb)\]
which is identified as the first stage of a spectral sequence converging to 
$H^*({\mathfrak{g}_\CC})$.

Under the hypothesis that $H^{2m}(\mathfrak{g}_\CC)=\CC$ (a condition that is satisfied 
for the Lie algebra of $G$ connected and compact, and also for $G$ nilpotent)
we obtain $\delb_\mub$-Hodge decomposition on $\cA^*_{\mathfrak{g}_\CC}$, giving an isomorphism
\[
H_{\Dol}^{p,q}(\mathfrak{g},J)\cong \Hh_{\delb_\mub}^{p,q}(\mathfrak{g}),
\]
where $\Hh_{\delb_\mub}^{p,q}(\mathfrak{g})$ is the subspace of $\cA^{p+q}_{\mathfrak{g}_\CC}$ of 
$\delb_\mub$-harmonic forms, defined after choosing a metric on $\mathfrak{g}$ compatible with $J$. 
This extends work of Rollenske \cite{Rollenske2} to the non-integrable setting (Theorem \ref{LiedelbmubHodge}).
In particular, the dimensions of $\Hh_{\delb_\mub}^{p,q}(\mathfrak{g})$ are metric-independent numbers.

Lie-algebra Dolbeault cohomology admits a translation to the geometric situation of compact Lie groups and nilmanifolds, respectively. We explain the latter situation.
Let $M= \Gamma \backslash G$ be a nilmanifold, where $G$ is a nilpotent Lie group with Lie algebra $\mathfrak{g}$, and $\Gamma$ is a co-compact subgroup acting on the left. Assume that 
$G$ carries a left $G$-invariant almost complex structure $J$, making $M$ into an almost complex manifold.
There is an inclusion $\cA^*_{\mathfrak{g}}\hookrightarrow \cA^*_{\dR}(M)$ which is a quasi-isomorphism by  Nomizu's Theorem \cite{Nomizu}.
We define the \textit{left-invariant Dolbeault cohomology} of $(M,J)$
by letting \[{}^L H_{\Dol}^{*,*}(M,J):=H^{*,*}_{\Dol}(\mathfrak{g},J).\]
Similarly, there is a \textit{left-invariant
Fr\"{o}licher spectral sequence} converging to the complex de Rham cohomology $H^*(M,\CC)$ of the manifold.
The harmonic theory for Lie algebras gives a description of
${}^L H_{\Dol}^{*,*}(M,J)$ in terms of $\delb_\mub$-harmonic left-invariant forms.
In general, even in the integrable case, it is not known if the inclusion of left-invariant forms into the algebra of forms of $M$ induces
an isomorphism on Dolbeault cohomology, although this is the case in several situations (see \cite{CFGU2}, \cite{ConFi}, \cite{Rollenske}, \cite{FRR}). In contrast, for maximally non-integrable structures, there are recent negative results in this direction. Indeed, while left-invariant cohomology is finite-dimensional by construction, 
in \cite{Stelzigandfriends} it is shown that Dolbeault cohomology for maximally non-integrable manifolds of dimension 4 or 6, turns out to be infinite-dimensional. Still, we always have an inclusion
\[{}^{L} H_{\Dol}^{*,*}(M,J)\hookrightarrow H_{\Dol}^{*,*}(M,J)\]
which, in some situations, proves to be sufficient in order to infer geometric results.

\subsection*{Maximally non-integrable and nearly K\"{a}hler manifolds}
Complex manifolds and maximally non-integrable manifolds are two endpoints in the
spectrum of almost complex manifolds. An almost complex structure is said to be \textit{maximally non-integrable} if the Nijenhuis
tensor \[N_J:T_x M  \otimes T_x M \longrightarrow T_xM\] has maximal rank at all points $x \in M$.
A main result below is that, for 4- and 6-dimensional manifolds, this condition implies degeneration of the 
Fr\"ohlicher spectral sequence at the second stage (Theorem \ref{46maxE2deg}). 

An important family of maximally non-integrable manifolds is given by (strictly) nearly K\"ahler $6$-manifolds, studied extensively by Gray \cite{Gray}, as well as \cite{Verbitsky2}, \cite{Verbitsky}, and others. In particular, 
the degeneration condition may be used to prohibit the existence of metrics 
for which the almost complex structure is nearly K\"ahler. 
The special local identities satisfied by nearly K\"ahler $6$-manifolds, together with the Harmonic Inclusion Theorem,
lead to other special properties of their Fr\"{o}licher spectral sequence.
We show by example that in general the Dolbeault cohomology contains strictly more information that the de Rham cohomology. 
For instance, the Lie group $SU(2)\times SU(2)$ admits a left-invariant nearly K\"{a}hler structure, with associated Hodge diamonds

\[
\xymatrix@C=.1pc@R=.1pc{
&&&1\\
&&3&&0\\
&0&&3&&0\\
0&&1&&1&&0\\
&0&&3&&0\\
&&0&&3\\
&&&1
}
\quad\quad\quad\quad
\xymatrix@C=.1pc@R=.1pc{
&&&1\\
&&0&&0\\
&0&&0&&0\\
0&&1&&1&&0\\
&0&&0&&0\\
&&0&&0\\
&&&1
}
\]
where the diamond on the left-hand side correspond to the Dolbeault numbers $h_{\Dol}^{*,*}$ and 
the one on the right to the Hodge numbers in de Rham cohomology $h_{\dR}^{*,*}$ obtained by Verbitsky \cite{Verbitsky}.
In this case, the left-invariant Fr\"{o}licher spectral sequence degenerates at the second term, and the above agree with the dimensions of $E_1$ and $E_2$ respectively.

In addition to the main $E_2$-degeneration result,  we show that the failure of first page, i.e. the Dolbeault cohomology, to agree with de Rham cohomology in certain bidegrees also prohibits the existence of nearly K\"ahler metrics compatible with a given almost complex structure. Indeed, we show that for any compact nearly K\"{a}hler 6-manifold  we have $h_{\Dol}^{2,1}=h_{\dR}^{2,1}$.

\section{Differential forms on almost complex manifolds}\label{SecPreliminars}
Let $M$ be an almost complex manifold of dimension $2m$. By definition, there is an endomorphism $J:TM\to TM$ such that $J^2=-1$ and the real tangent space $T_xM$ at $x\in M$ has a complex structure $J_x$. The 
complexification decomposes into $+i$ and $-i$ eigenspaces $T^{1,0}_xM$ and $T^{0,1}_xM$ respectively:
\[T_xM\otimes\CC= T^{1,0}_xM\oplus T^{0,1}_xM.\]
By taking duals and exterior powers, this decomposition gives a bigrading on the algebra of differential forms with values in $\CC$:
\[\cA^n(M):=\cA_{\dR}^n(M)\otimes_\R \CC=\bigoplus_{p+q=n}\cA^{p,q}.\]
As an abbreviated notation, we denote the fiber of forms at $x \in M$ by 
\[
\cA^{p,q}_x := \bigwedge^p \left(T^*_x M \ot \CC\right)^{1,0}  \otimes  \bigwedge^q \left(T^*_x M \ot \CC \right)^{0,1}
\]
Since $\cA^{p,q}$ is generated by $\cA^{0,0}$, $\cA^{0,1}$ and $\cA^{1,0}$, one may verify that for all $p,q\geq 0$, 
the exterior derivative satisfies
\[d(\cA^{p,q})\subseteq \cA^{p-1,q+2} \oplus \cA^{p,q+1} \oplus \cA^{p+1,q} \oplus \cA^{p+2,q-1}.\]
Therefore, we may write
$d=  \mub+\delb+\del+\mu$
where the bidegrees of each component are given by 
\[
|\mub|=(-1,2), \,  |\delb|=(0,1), \, |\del|=(1,0), \text{ and } \,  |\mu|=(2,-1).
\]
The components $\delb$ and $\mub$ are complex conjugate to $\del$ and $\mu$ respectively.

Expanding the equation $d^2 =0$ we obtain the following set of equations:
\newcommand\numberthis{\addtocounter{equation}{1}\tag{$\vartriangle$}}
\begin{align*}
\mu^2 = 0 \\
\mu \del + \del \mu =  0 \\
\mu \delb + \delb \mu +\del^2 =  0 \\
\mu \mub + \del \delb + \delb \del + \mub \mu  = 0  \numberthis \label{eq;drelations}\\
\mub \del + \del \mub +\delb^2 =  0 \\
\mub \delb + \delb \mub =  0 \\
\mub^2 = 0
\end{align*}

Also, by expanding the Leibniz rule
\[d(\omega \wedge \eta)=d\omega \wedge \eta + (-1)^{|\omega|}\omega \wedge d\eta \]
we obtain a Leibniz rule for each component of $d$. Note that by
degree consideration, $\mub$ and $\mu$ vanish on functions, so each are linear over functions, and therefore give well defined fiberwise maps, $\mub: \cA_x^{p,q} \to \cA_x^{p-1,q+2}$, and similarly for $\mu$. 
In particular, both $(\cA^{*,*},\mub)$ and $(\cA^{*,*},\mu)$ are commutative differential bigraded algebras.

The integrability theorem of Newlander and Nirenberg \cite{NN} states that the almost complex structure $J$
is integrable if and only if $N_J\equiv 0$, where
\[N_J : TM \ot TM \longrightarrow TM\] denotes the Nijenhuis tensor 
\[
N_J(X,Y) := [X,Y]+ J[X,JY] + J[JX,Y]-[JX,JY].
\]
The following well-known result shows that 
$J$ is integrable if and only if $\mub\equiv 0$. We provide a short proof for convenience of the reader.
 
\begin{lem} \label{lem;Nmub}
With the above notation, we have:
\[
\mu + \mub = - \frac 1 4 \left( N_J \ot \textrm{Id}_\mathbb{C} \right)^*,
\]
where the right hand side has been extended over all forms as a derivation. 
\end{lem}

\begin{proof} Let $\pi_{1,0}$ and $\pi_{0,1}$ denote the projections onto $T^{1,0}M$ and $T^{0,1}M$ respectively, and let 
$\pi^{1,0}$ and $\pi^{0,1}$ be the dual projections onto the respective co-tangent spaces. Let $\pi^{2,0}$ be projection onto $(2,0)$-forms. For $\omega$ a $(0,1)$-form, and any real vectors $X$ and $Y$, we use that $\omega$ vanishes on $(1,0)$-vectors to compute
\begin{eqnarray*}
\left( \mu \, \omega \right) (X,Y) &=& \pi^{2,0} d \, \omega (X,Y) \\
&=& d \omega (\pi_{1,0} X, \pi_{1,0} Y ) \\
&=& -\omega ( \pi_{0,1} [\pi_{1,0} X, \pi_{1,0} Y ] ) \\
&=& - \frac 1 4 \omega ( N_J(X,Y)).
\end{eqnarray*}
The first equality is by definition and the third uses Cartan's formula 
\[d \omega(X,Y) = X \omega(Y) - Y \omega(X) - \omega( [X,Y]).\]
The last equality follows from a straightforward calculation using 
\[\pi_{1,0} = \frac 1 2 \left(Id - i J \right),\, \pi_{0,1} = \frac 1 2 \left( Id + i J \right),\]
the Nijenhuis tensor, 
and again the fact that $\omega$ vanishes on $(1,0)$-vectors.

Similarly we have $\mub \, \eta = -\frac 1 4 \left( N_J \otimes \textrm{Id}_\mathbb{C} \right)^*$ for $\eta$ a $(1,0)$-form. Since $\mu$ vanishes on $(1,0)$-forms, and 
$\mub$ vanishes on $(0,1)$-forms, the lemma follows.
\end{proof}

\begin{defn}
A \textit{morphism of almost complex manifolds} $f:M\to M'$ is given by a differentiable map
compatible with the almost complex structures, i.e., for all $x \in M$ the following diagram commutes:
\[
\xymatrix{
T_x{M}\ar[d]^{J_x}\ar[r]^{f_*}&T_{f(x)}M'\ar[d]^{{J'_{f(x)}}}\\
T_x{M}\ar[r]^{f_*}&T_{f(x)}M'
}
\]
An \textit{isomorphism} of almost complex manifolds is a morphism of almost complex manifolds which is also a diffeomorphism.
\end{defn}

Note that any morphism of almost complex manifolds $f:M\to M'$ induces a morphism of differential graded algebras
$f^*:\cA({M'})\to \cA(M)$
 such that:
\begin{enumerate}
 \item[(i)] $f^*$ preserves the $(p,q)$-decompositions: $f^*(\cA^{p,q}_{M'})\subseteq \cA^{p,q}_{M}$.
 \item[(ii)] $f^*$ is compatible with $d$ component-wise:
if $\delta$ is any of the components $\mub$, $\delb$, $\del$ or $\mu$ of the differential $d$, then
$f^*\delta =\delta f^*$.
\end{enumerate}

\section{Fr\"{o}licher spectral sequence}\label{SecFro}
Throughout this section, let $M$ be an almost complex manifold and denote by $(\cA=\bigoplus \cA^{*,*},d)$ its complex de Rham algebra. 

By using the last three of the equations in (\ref{eq;drelations}) of Section \ref{SecPreliminars}, we will arrive at our definition of \emph{Dolbeault cohomology}.
First, the equation $\mub^2=0$ of (\ref{eq;drelations}) gives well-defined vector spaces 
\[H_\mub^{p,q}(M):={\Ke(\mub:\cA^{p,q}\longrightarrow \cA^{p-1,q+2})\over \Img(\mub:\cA^{p+1,q-2}\longrightarrow \cA^{p,q})}.
\]
Next, from the equation $\delb\mub+\mub\delb=0$, it follows that $\delb$ induces a well-defined map
$\delb:H^{p,q}_\mub\to H^{p,q+1}_\mub$. Finally, the equation $\mub \del + \del \mub +\delb^2=0$ shows $\delb^2$ is chain homotopic to zero, with respect to the differential $\mub$, and the chain homotopy $\del$.

\begin{defn} \label{defnDol}
Define the \textit{Dolbeault cohomology of $M$} by  
\[H_{\Dol}^{p,q}(M):=H^{q}(H_\mub^{p,*},\delb)={\Ke(\delb:H_\mub^{p,q}(M)\longrightarrow H_\mub^{p,q+1}(M))\over
\Img(\delb:H_\mub^{p,q-1}(M)\longrightarrow H_\mub^{p,q}(M))
}.\]
\end{defn}

We next show that the Dolbeault cohomology of every almost complex manifold arises with the first stage of a functorial spectral sequence converging to the complex de Rham cohomology. For this, we modify the classical Hodge filtration by taking into account the presence of $\mub$ and making it compatible with the total differential.

\begin{defn}\label{Hodgefildef}
Define the \textit{Hodge filtration} of $\cA$ as the decreasing filtration $F$ given by
\[F^p\cA^n:=\Ke(\mub)\cap \cA^{p,n-p}\oplus\bigoplus_{i> p}\cA^{i,n-i}.\]
\end{defn}

Note that in the integrable case, $F$ coincides with the usual Hodge filtration.
\begin{lem}
The Hodge filtration $F$ makes $(\cA,d,F)$ into a filtered dg-algebra with 
\[F^{n+1}\cA^n=0\text{ and }F^0\cA^n=\cA^n\text{ for all }n\geq 0.\]
\end{lem}
\begin{proof}
Using the fact that
\[d(\cA^{p,q})\subseteq \cA^{p-1,q+2} \oplus \cA^{p,q+1} \oplus \cA^{p+1,q} \oplus \cA^{p+2,q-1}\]
we first note that
\[d(\Ke(\mub)\cap \cA^{p,n-p})\subseteq
\delb(\Ke(\mub)\cap \cA^{p,n-p})\oplus \cA^{p+1,n-p}\oplus\cA^{p+2,n-1-p}.
\]
Since $\delb\mub+\mub\delb=0$ the first direct summand satisfies 
\[\delb(\Ke(\mub)\cap \cA^{p,n-p})\subseteq \Ke(\mub)\cap \cA^{p,n+1-p}.\]
Now, note that 
\[d(\bigoplus_{i\geq p+1}\cA^{i,n-i})\subseteq \mub(\cA^{p+1,n-p-1})\oplus \bigoplus_{i\geq p+1}\cA^{i,n-i}.\]
Since $\mub^2=0$ the first direct summand satisfies
\[\mub(\cA^{p+1,n-p-1})\subseteq \Ke(\mub)\cap \cA^{p,n+1-p}.\]
This proves that $F$ is compatible with the differential and hence $(\cA,d,F)$ is a filtered cochain complex.
Compatibility of $F$ with the algebra structure is straightforward.
The boundedness conditions follow from the fact that 
 $\mub$ is trivial on $\cA^{0,*}$.
\end{proof}

\begin{defn} \label{FSS}
The \textit{Fr\"{o}licher spectral sequence} $\{E_r(M),\delta_r\}_{r \geq 0}$ of $M$ is the spectral sequence
$E_r(M):=E_r(\cA,F)$
associated to the filtered dg-algebra $(\cA,d,F)$.
\end{defn}
For each $r\geq 0$, the pair $(E_r(M),\delta_r)$ is a commutative differential $r$-bigraded algebra.
Since $F$ is bounded, this spectral sequence 
converges to the complex de Rham cohomology \[H^*_{\dR}(M,\mathbb{C})=H_{\dR}^*(M,\R)\otimes_\R\mathbb{C}.\]

It will be useful to consider an alternative filtration
whose spectral sequence is related to that of $F$ by a shift of indexing (c.f. Remark \ref{compareFs}).
It is defined as follows:
\begin{defn} \label{shFSS}
Define the \textit{shifted Hodge filtration} of $\cA$ as the decreasing filtration $\widetilde{F}$ given by
\[\wt F^p\cA^n:=\bigoplus_{i\geq p-n}\cA^{i,n-i}.\]
\end{defn}

\begin{lem}\label{comparison_ss}
The shifted Hodge filtration  makes $(\cA,d,\widetilde{F})$ into a filtered dg-algebra such that 
\[\wt{F}^{2n+1}\cA^n=0\text{ and } \wt{F}^{n}\cA^n=\cA^n.
\]
For $r\geq 0$ we have canonical morphisms of differential bigraded algebras
\[E_{r}^{p,n-p}(\cA,F)\longrightarrow E_{r+1}^{p+n,-p}(\cA,\widetilde F)\]
which are isomorphisms for all $r\geq 1$.
\end{lem}
\begin{proof}
We have:
\begin{multline*}
d(\wt{F}^p\cA^n)=\bigoplus_{i\geq p-n}d \cA^{i,n-i}\subseteq \\
\subseteq
\bigoplus_{i\geq p-n}
\left(\cA^{i-1,n-i+2}\oplus\cA^{i,n-i+1}\oplus\cA^{i+1,n-i}\oplus\cA^{i+2,n-i-1}\right)\subseteq \\ 
\subseteq\bigoplus_{i\geq p-n}\cA^{i-1,n-i+2}=\bigoplus_{i\geq p-n-1}\cA^{i,n+1-i}=\wt{F}^p\cA^{n+1}.
\end{multline*}
This proves that $\widetilde{F}$ is compatible with the differential and hence $(\cA,d,\widetilde{F})$ is a filtered cochain complex. Compatibility of $F$ with the algebra structure is straightforward.

Now, note that the filtrations $F$ and $\widetilde F$ are related via Deligne's d\'{e}calage filtration:
\[F^p\cA^n=\mathrm{Dec} \widetilde{F}^p\cA^n:=\{\omega\in \widetilde{F}^{p+n}\cA^n; d\omega\in \widetilde{F}^{p+n+1}\cA^{n+1}\}.
\]
Thus the lemma follows from Proposition I.3.4 of \cite{DeHII}, stating that the 
spectral sequences associated with a filtration and its d\'{e}calage are related by a shift of indexing as above.
\end{proof}

\begin{rmk}\label{multicomplex}
Note that the shifted Hodge filtration makes $\cA$ into a bigraded multicomplex by letting $\widetilde \cA^{p,n-p}:=\cA^{p-n,2n-p}$. Then by definition we have that 
\[\widetilde F^p\cA^n=\bigoplus_{i\geq p} \widetilde \cA^{i,n-i},\]
so $\widetilde F$ is the column filtration associated to this multicomplex.
With this new bigrading, the components $\mub$, $\delb$, $\del$ and $\mu$ of $d$ have bidegrees 
\[|\mub|'=(0,1), |\delb|'=(1,0), |\del|'=(2,-1)\text{ and }|\mu|'=(3,-2).\]
So the shifted Hodge filtration just rotates the bidegrees of the components of $d$
with respect to the initial grading given by the Hodge filtration.
\end{rmk}

The Fr\"{o}licher spectral sequence admits a very explicit description in terms of the components 
$\mub$, $\delb$, $\del$ and $\mu$ of the differential, whose higher terms we detail in the Appendix.
Here we just describe the first stage.

\begin{thm}\label{E1HDol}
Let $M$ be an almost complex manifold. We have isomorphisms
\[H_{\Dol}^{p,q}(M)\cong E_1^{p,q}(M)\cong 
{\left\{ \omega \in \cA^{p,q}\cap\Ke(\mub); \delb  \omega  \in\Img(\mub)\right\}
\over{
\left\{  \omega  \in \cA^{p,q}\text{ such that }  \omega  =\mub \alpha +\delb \beta \text{ with } \mub \beta =0\right\}
}}
\]
and the differential $\delta_1:E_1^{p,q}(M)\to E_1^{p+1,q}(M)$  is given by
\[\delta_1[ \omega  ]=[\del   \omega  -\delb \eta],\]
where $\eta$ is any form in $\cA^{p+1,q-1}$ satisfying $\mub \eta = \delb \omega$.
\end{thm}
\begin{proof}
By Lemma \ref{comparison_ss}, we have 
\[E_1^{p,n-p}(M):=E_1^{n-p}(\cA,F)\cong E_2^{p+n,-p}(\cA,\widetilde F),\]
where $\widetilde F$ is the shifted Hodge filtration.
Therefore it suffices to identify $E_2^{p+n,-p}(\cA,\widetilde F)$ with $H_{\Dol}^{p,q}(M)$ and describe 
the differential 
\[\delta_1=\widetilde \delta_2:E_2^{*,*}(\cA,\widetilde F)\longrightarrow E_2^{*+2,*-1}(\cA,\widetilde F).\]
As noted in Remark \ref{multicomplex}, the filtration $\widetilde F$ is the column filtration of a multicomplex $(\widetilde A^{*,*},d=d_0+d_1+d_2+d_3)$ whose differential has four components
of bidegree $|d_i|=(i,1-i)$, where $\widetilde \cA^{p,n-p}:=\cA^{p-n,2n-p}$ and
$d_0=\mub$, $d_1=\delb$, $d_2=\del$ and $d_3=\mu$.
Each term of its associated spectral sequence as well as a formula for the induced differentials 
\[\widetilde \delta_i:E_i^{*,*}(\cA,\widetilde F)\longrightarrow E_i^{*+i,*-i+1}(\cA,\widetilde F)\] 
have a known description, given by Livernet-Whitehouse-Ziegenhagen in \cite{LW}. We obtain:
\[E_0^{p+n,-p}(\cA,\widetilde F)\cong \cA^{p,n-p}\text{ and }\widetilde \delta_0=\mub.\]
The first stage is given by 
\[E_1^{p+n,-p}(\cA,\widetilde F)\cong H_\mub^{p,q}(M)\text{ and }\widetilde \delta_1[ \omega ]=[\delb  \omega ].\]
In particular, we find that 
\[E_2^{p+n,-p}(\cA,\widetilde F)\cong H^{p+n,-p}(E_1^{*,*}(\cA,\widetilde F),\widetilde\delta_1)\cong H^{q}(H_\mub^{p,*},\delb)\cong H_{\Dol}^{p,q}(M).\]
Lastly, the description of $E_2^{p+n,-p}(\cA,\widetilde F)$ and  $\delta_1=\widetilde \delta_2$
appearing in \cite{LW}, directly gives the formulas in the statement of the theorem. 

In fact, the formula for the first and second stages of the spectral sequence of a multicomplex coincides 
with those of a bicomplex. Therefore the description of $H_{\Dol}^{*,*}(M)$ coincides with the description of the $E_2$-term of the classical spectral sequence defined by Fr\"{o}licher in \cite{Fro}, after replacing the roles of $\delb$ and $\del$ by $\mub$ and $\delb$.
\end{proof}

\begin{rmk} \label{compareFs}
The shifted Hodge filtration $\widetilde F$ in some sense contains more information than the Hodge filtration $F$. Indeed, we have $F=\mathrm{Dec} \widetilde F$, but in general, Deligne's d\'{e}calage functor $\mathrm{Dec}$ does not have an inverse.
Our presentation of Dolbeault cohomology as the first stage of the spectral sequence associated to $F$ (instead of
presenting it as the second stage of the spectral sequence associated to $\widetilde F$) is mainly to 
naturally recover the usual Fr\"{o}licher spectral sequence in the integrable case, as well as to preserve the original bidegrees of the components $\mub$, $\delb$, $\del$ and $\mu$ of the total differential.
However, in certain situations it may be useful to consider the shifted spectral sequence. For instance, note that 
Lemma 
\ref{ss_functorial} below shows that the first term
\[(E_1^{*,*}(\cA,\widetilde F),\widetilde \delta_1)\cong (H_\mub^{*,*}(M),\delb)\]
is a well-defined invariant of the almost complex manifold $M$.
In analogy with the integrable setting, one may call the quasi-isomorphism
type of this differential bigraded algebra, the \textit{Dolbeault homotopy type of $M$}. By Lemma \ref{comparison_ss}, 
the algebra $(H_\mub^{*,*}(M),\delb)$ is quasi-isomorphic to the algebra $(E_0^{*,*}(M),d_0)$, but the description of the former one retains the original geometry.
\end{rmk}

The following Lemma shows that the Fr\"{o}licher spectral sequence, and in particular $H_{\Dol}^{*,*}(M)$, is a well-defined invariant for almost complex manifolds.

\begin{lem}\label{ss_functorial}
Let $f:M\to M'$ be a morphism of almost complex manifolds. Denote by $\cA$ (resp. $\cA'$) the complex de Rham algebra of $M$ (resp. $M'$).
Then $f^*:\cA'\to \cA$ is compatible with both $\widetilde{F}$ and $F$:
\[f^*(\widetilde{F}^p\cA')\subseteq \widetilde{F}^p\cA\text{ and }f^*(F^p\cA')\subseteq F^p\cA.\]
In particular, for all $r\geq 0$ it induces morphisms of spectral sequences 
 \[E_r(f^*):E_r(\cA',\widetilde{F})\to E_r(\cA,\widetilde{F})\text{ and }E_r(f^*):E_r(\cA',{F})\to E_r(\cA,{F})\]
 and hence a morphism between Dolbeault cohomologies
 \[f^{*}:H_{\Dol}^{p,q}(M')\longrightarrow H_{\Dol}^{p,q}(M).\]
\end{lem}
\begin{proof}
Since $f^*$ preserves bigradings,
it is compatible with $\widetilde{F}$. Since $f^*\mub =\mub f^*$
we have that $f^*(\Ke(\mub))\subseteq \Ke(\mub)$. Therefore $f^*$ is compatible with $F$.
\end{proof}

This immediately gives:
\begin{cor} \label{HdRpq}
The complex de Rham cohomology of an almost complex manifold has a bigrading,
given by the $E_\infty$-page of the 
Fr\"{o}licher spectral sequence.
This bigrading is functorial and generalizes the 
existing bigrading on the complex de Rham cohomology of complex manifolds.
\end{cor}

\begin{rmk}\label{inspection}
By direct inspection of the Fr\"{o}licher spectral sequence (whose higher stages are detailed in the Appendix)  we may gain some insight by studying the horizontal bottom row $q=0$ and the left vertical  column $p=0$. We have 
\[
H_{\Dol}^{p,0}(M)=E_1^{p,0}(M)=\cA^{p,0}\cap \Ke(\mub)\cap \Ke(\delb)\]
and
\[E_r^{p,0}(M)\cong {
\left\{ \omega \in \cA^{p,0}\cap \Ke(d)\right\}
\over
\left\{
 \omega =d(\eta_0+\cdots+\eta_{r-2}); \eta_i\in \cA^{p-1-i,i}
\right\}}, \text{ for }r\geq 2.
\]
In particular, if $r\geq 2$ and $r\geq p+1$ we have
\[
E_r^{p,0}(M)=E_\infty^{p,0}(M)\cong
{
\cA^{p,0}\cap\Ke(d)
\over
\cA^{p,0}\cap \Img(d)
},
\]
where for the higher terms we used the formula for $E_r$ given in the Appendix.
This gives
\[E_2^{0,0}(M)\cong \cA^{0}\cap\Ke(d)\cong H^0(M;\CC).\]
Also, the above expressions give a sequence of surjective morphisms
\[E_2^{p,0}(M)\twoheadrightarrow E_3^{p,0}(M)\twoheadrightarrow\cdots \twoheadrightarrow E_r^{p,0}(M)\twoheadrightarrow\cdots.\]
Likewise, for the vertical left column, we have a sequence of injections
\[\cdots\hookrightarrow E_r^{0,q}(M)\hookrightarrow \cdots \hookrightarrow E_2^{0,q}(M)\hookrightarrow E_1^{0,q}(M).\]
\end{rmk}

For compact manifolds, the bottom right corner
of the spectral sequence has the following special property, which in particular implies that classes in $H_{\Dol}^{m,0}(M)$ have well defined periods on homology classes in degree $m$ (c.f. Lemma 5 of \cite{CFG} in the integrable case).
\begin{lem}
Let $M$ be a compact almost complex manifold of dimension $2m$.
Then for all $r\geq 1$ we have 
\[H_{\Dol}^{m,0}(M)=E_r^{m,0}(M)\cong E_\infty^{m,0}(M)\cong \cA^{m,0}\cap \Ke(d).\]
In particular, if $H^m(M,\CC)=0$ then $H_{\Dol}^{m,0}(M)=0$.
\end{lem}
\begin{proof}
If $[ \omega ]\in E_r^{m,0}(M)$ then $d \omega =0$. It then suffices to prove that if $ \omega =d \eta$ then $[\eta]=0$.
By Stoke's Theorem, 
\[
\int_M  \omega  \wedge \bar  \omega  = \int_M d( \eta \wedge d \bar \eta) = 0,
\]
so that $ \omega =0$, since the pairing $(\alpha,\beta) = (\sqrt{-1})^{m^2} \int_M \alpha \wedge \bar \beta$ defines a positive definite inner product on $\cA^{m,0}$.
\end{proof}

\section{Harmonic theory}\label{SecHarm}

It is well known that for a compact complex manifold
with Hermitian metric, Dolbeault cohomology is isomorphic to the space of  $\delb$-harmonic forms,
defined as the kernel of the $\delb$-Laplacian $\Delta_\delb := \delb \delb^* + \delb^* \delb$.
This follows from the Hodge decomposition theorem and the fact that $\Delta_\delb$ is elliptic.
In fact, the operator $\Delta_\delb$ is elliptic on any almost Hermitian manifold; the argument showing the symbol is an isomorphism does not involve integrability, so the Hodge decomposition with respect to $\delb$
is true in the non-integrable setting:

\begin{thm}[Hodge $\delb$-decomposition]\label{delbHodge}
Given a compact almost Hermitian manifold, for any bidegree $(p,q)$ let
\[\Hh^{p,q}_\delb:= \Ke(\Delta_\delb) \cap \cA^{p,q}.\]
Then for every bidegree $(p,q)$ the following hold:
\begin{enumerate}
 \item There are orthogonal direct sum decompositions
\[\cA^{p,q}=\Hh^{p,q}_\delb \oplus \Delta_\delb(\cA^{p,q})\]
\item Every element $\omega \in \cA^{p,q}$ can be written uniquely as 
\[\omega =\Hh_\delb(\omega )+\Delta_\delb G_\delb (\omega ) = \Hh_\delb(\omega )+ G_\delb \Delta_\delb (\omega )  \]
where $\Hh_\delb(\omega )$ denotes the projection of $\omega $ to the space $\Ke(\Delta_\delb)$ of $\delb$-harmonic forms, 
and $G_\delb$ denotes the corresponding Green's operator.
\item The space $\Hh_\delb$ is orthogonal to $\Img(\delb)$ and $\Img(\delb^*)$, and is finite-dimensional.
\end{enumerate}
\end{thm}

In contrast with the integrable case, the Green operator $G_\delb$ may not commute with $\delb$ and $\delb^*$, and the image of $\delb^*$ may not be orthogonal to the image of $\delb$.  Additionally, in the general case the image of $\Delta_\delb$ is closed, whereas in the integrable case each of the images $\delb$ and $\delb^*$ are closed.
Note however, that $\Delta_\delb$ and $G_\delb$ commute, and that both images of $\delb$ and $\delb^*$ are orthogonal to $\Hh_\delb$.

Let $\Mh$ be an almost Hermitian manifold. The compatible metric determines a metric  on the fibers $\cA^{p,q}_x$ of the $(p,q)$-bundles, and an associated Hodge-star operator \[\star: \cA^{p,q}_x \longrightarrow \cA^{m-q,m-p}_x\quad\text{defined by}\quad
\omega \wedge \star \bar \eta = \langle \omega , \eta \rangle \Omega,
\] 
where $\Omega$ is the volume form determined by the Hermitian metric. Composing $\star$ with conjugation gives another isomorphism $\bar \star: \cA^{p,q}_x \to \cA^{m-p,m-q}_x$.

 For any of the operators $\delta =  \mub, \delb, \del, \mu,$ and $d$, there are operators 
 \[
 \delta^* : = - \star \bar \delta \star, 
\]
and we define the \textit{$\delta$-Laplacian} by
\[
\Delta_\delta:= \delta \delta^*+\delta^*\delta,
\]
which satisfies 
\[
\star \Delta_{\bar \delta} = \Delta_{ \delta} \star .
\]

We denote the space of \textit{$\delta$-harmonic forms} by
\[\Hh_\delta^{p,q}:= \Ke(\Delta_\delta) \cap \cA^{p,q} =  \Ke(\delta)\cap \Ke(\delta^*) \cap \cA^{p,q}.\]
Note that for $\delta = \mub$ and $\delta = \mu$, this is well defined fiberwise by 
\[
\Hh_{\delta,x}^{p,q}:= \Ke(\Delta_\delta) \cap \cA^{p,q}_x = \Ke(\delta)\cap \Ke(\delta^*) \cap \cA^{p,q}_x.
\]
It is well known that for a closed manifold the operator $\delb^*$ is equal to the  $\mathcal{L}_2$-adjoint of $\delb$. The next lemma shows that a point-wise version of this statement holds for $\mub$, which implies
$\mub^*$ is equal to the $\mathcal{L}_2$-adjoint of $\mub$.

\begin{lem} \label{mub*adjoint}
For any almost Hermitian manifold, the operator $\mub^*$ is equal to the fiberwise metric adjoint of $\mub$.
\end{lem}

\begin{proof} On an even dimensional manifold, we have 
\[\star \alpha \wedge \beta = (-1)^{|\alpha|} \alpha\wedge \star \beta \text{ and }
\star^2 \alpha = (-1)^{|\alpha|} \alpha.\]
For $\omega \in \cA^{p-1,q+2}$ and $\eta \in \cA^{p,q}$ we compute 
\[
\langle \star \mu \star \omega , \eta \rangle   \Omega \, = \mu \star \omega \wedge  \bar \eta  =
 (-1)^{p+q}\star \omega \wedge  \mu \bar \eta  = 
  -  \omega  \wedge \star  \mu \bar \eta   = -\, \langle \omega, \mub \eta \rangle \Omega,
\]
where  in the second equality we use the fact that $\mu$ is a derivation, so that 
\[
\mu \star \omega \wedge  \bar \eta  = 
  \mu (\star \omega \wedge \bar \eta) -(-1)^{2m-p-q-1}   \star \omega  \wedge \mu \bar \eta  =  (-1)^{p+q}\star \omega  \wedge \mu \bar \eta,
 \] 
 since $\star \omega \wedge \bar \eta \in \cA^{m-2,m+1}$ vanishes for degree reasons. 
\end{proof}

We next show that there is a fiberwise Hodge decomposition determined by $\mub$ on any almost Hermitian manifold.  This decomposition becomes global when the operator $\mub$ has no jumps, as we will prove in subsection \ref{LCR}.

\begin{lem}[Fibrewise $\mub$-Hodge decomposition]\label{mubHodge}Let $\Mh$ 
be an almost Hermitian manifold.
For every bidegree $(p,q)$ the following hold:
\begin{enumerate}
 \item There are direct sum decompositions
\[\cA^{p,q}_x
= \mub (\cA^{p+1,q-2}_x) \oplus  \Hh_{\mub,x}^{p,q} \oplus \mub^*(\cA^{p-1,q+2}_x).\]
\item Every element $\omega\in \cA^{p,q}_x$ can be written uniquely as 
\[\omega=\Hh_{\mub,x}(\omega)+\Delta_\mub G_\mub (\omega)\]
where $\Hh_{\mub,x}(\omega)$ denotes the projection of $\omega$ to the space $\Hh_{\mub,x}^{p,q}$ of $\mub$-harmonics in 
$\cA^{p,q}_x$, and $G_\mub$ denotes the corresponding Green's operator.
\item  The map $\omega\mapsto [\omega]$ defines an isomorphism
\[\Hh_{\mub,x}^{p,q}\cong H_{\mub,x}^{p,q}(M):={{\Ke(\mub:\cA^{p,q}_x\longrightarrow \cA^{p-1,q+2}_x)}\over{\mathrm{\Img}(\mub:\cA^{p+1,q-2}_x\longrightarrow\cA^{p,q}_x)}}.\]
In particular,
every element in $H_{\mub,x}^{p,q}(M)$ has a unique $\mub$-harmonic representative.
\end{enumerate}
\end{lem}
\begin{proof}
Since $\mub$ is a linear operator on a finite-dimensional vector space $\cA_x^{p,q}$,
there are orthogonal direct sum decompositions
\[\cA_x^{p,q}=\left( \Ke(\mub) \cap  \cA_x^{p,q}\right) \oplus \mub^* \cA_x^{p-1,q+2}=
\mub \cA_x^{p+1,q-2} \oplus \left( \Ke(\mub^*) \cap   \cA_x^{p,q} \right).
\] 
Since $\mub^2=0$, $\Img(\mub)\subseteq \Ke (\mub)$, and since $\Hh_{\mub} =\Ke(\mub)\cap \Ke(\mub^*)$,
the orthogonal complement of $\Imt (\mub) \cap  \cA_x^{p,q}$ in $\Ke(\mub) \cap  \cA_x^{p,q}$ is $\Hh_{\mub,x}^{p,q}$.
This proves (1), and (2) follows, where the Green's operator is given by orthogonal projection onto $\Imt (\mub^*)$ composed with the inverse of the isomorphism $\mub :  \Imt(\mub^*) \to \Imt (\mub)$.
The last statement follows readily.
\end{proof}

We next define a space of harmonic forms which will be shown to inject into the Dolbeault cohomology groups.

\begin{defn} The space of \textit{$\delb$-$\mub$-harmonic forms} of an almost Hermitian manifold is given by
\[
\Ke \left(\Delta_\delb + \Delta_\mub \right) \cap \cA^{p,q}  =  \Hh_{\delb}^{p,q} \cap  \Hh_{\mub}^{p,q}.
\]
\end{defn}

The equality in the definition can easily be checked by expanding 
\[
\langle ( \Delta_\delb + \Delta_\mub ) \omega , \omega \rangle  =0.
\]
Note that the spaces are finite-dimensional whenever the manifold is compact. This follows since $\Hh_{\delb}^{p,q} $ is finite-dimensional by Theorem \ref{delbHodge}. These spaces satisfy Serre duality.

\begin{prop}[Serre duality]\label{SerreDuality}
Let $\Mh$ be an almost Hermitian manifold of dimension $2m$. For every bidegree $(p,q)$, and all $x$, we have isomorphisms
\[
\Hh_{\mub,x}^{p,q} \cong \Hh_{\mub,x}^{m-p,m-q} \quad \textrm{and} \quad \Hh_{\delb}^{p,q}\cong \Hh_{\delb}^{m-p,m-q},
\]
as well as isomorphisms
\[ 
\Hh_{\delb}^{p,q} \cap \Hh_{\mub}^{p,q}  \cong \Hh_{\delb}^{m-p,m-q} \cap \Hh_{\mub}^{m-p,m-q}.
\]
\end{prop}

\begin{proof}
The fact that $\star \Delta_{\mu} = \Delta_\mub \star$ implies 
\[
\star: \Hh_{\delta,x}^{p,q} \cong \Hh_{\mu,x}^{m-q,m-p}
\]
 is an isomorphism. Composing this with conjugation gives an isomorphism 
 \[
 \bar \star : \Hh_{\mub,x}^{p,q} \cong \Hh_{\mub,x}^{m-p,m-q}.
 \]
The other two statements follow similarly using $\star \Delta_{\del} = \Delta_\delb \star$. 
\end{proof}

\begin{lem}[Harmonic Inclusion]
\label{HIT}
Let $\Mh$ be an almost Hermitian manifold. The identity map induces an injection
\[
\Hh_\delb^{p,q} \cap \Hh_{\mub}^{p,q} \,  \subseteq  \,
H_{\Dol}^{p,q}(M).
\]
\end{lem}
\begin{proof}
We use the expression for Dolbeault cohomology given in Theorem \ref{E1HDol}. Let $\omega\in \Hh_\delb^{p,q}\cap\Hh_\mub^{p,q}$. Since $\delb \omega=0$ and $\mub \omega=0$, this gives a well-defined class $[w]$ in 
\[H^{p,q}_{\Dol}(M)\cong  
{\left\{ \omega \in \cA^{p,q}\cap\Ke(\mub); \delb  \omega  \in\Img(\mub)\right\}
\over{
\left\{  \omega  \in \cA^{p,q}\text{ ; }  \omega  =\mub \alpha +\delb \beta \text{ ; } \mub \beta =0\right\}
}}.
\]
Assume that $[w]=0$. Then $\omega=\mub \alpha +\delb \beta$ with 
$\mub \beta =0$. It follows that 
\[||\omega||^2=\langle w,\mub \alpha +\delb \beta\rangle=\langle \omega,\mub \alpha\rangle+\langle \omega,\delb \beta\rangle=
 \langle \mub^*\omega,\alpha\rangle+\langle \delb^*\omega,\beta\rangle=0
\]
where we used the fact that, since $\omega$ is $\delb$-$\mub$-harmonic, we have $\delb^* \omega=0$ and $\mub^*\omega=0$. This gives $\omega=0$ and so the assignement $\omega\mapsto [\omega]$ is injective.
\end{proof}

\begin{rmk}
By analogy with Hirzebruch's Problem 20, one can ask in the compact case whether the numbers 
$\dim \left(\Hh_\mub^{p,q}\cap \Hh_\delb^{p,q}\right)$ are metric-independent. The Theorem shows these are bounded by metric-independent numbers:
\[
\dim \left(\Hh_\delb^{p,q}\cap \Hh_\mub^{p,q} \right) \leq \dim H_{\Dol}^{p,q}(M).
\]
 For compact complex manifolds, we have equality by the $\delb$-Hodge decomposition.
\end{rmk}

\subsection{Finiteness results}
In this subsection we use harmonic theory to give some finiteness results for $H_{\Dol}^{*,*}(M)$ when $M$ is compact. First, we show $H_{\Dol}^{0,0}(M) \cong \CC$ when $M$ is compact and connected. For this we need the following Lemma, which is a special case of the K\"ahler identities that holds for all $1$-forms on almost complex manifolds (c.f. \cite{Osh}, Appendix).
In \cite{CWAH} this is generalized to forms of all degrees.

\begin{lem} \label{01andf} For any metric $\langle \,,\, \rangle$ compatible with $J$, let $\omega(X,Y) = \langle J X, Y  \rangle$, let $L$ denote be the Lefschetz operator given by $L(\eta) = \omega \wedge \eta$, and let $\Lambda$ denote the adjoint of $L$. 
Then, 
for all $\alpha \in \cA^{1,0}$ we have
\[
 \Lambda \del \alpha = i \delb^* \alpha + i [\Lambda, \delb^*] L  \alpha.
\] 
In particular, for any function $f: M \to \CC$, we have
\[
-i \Lambda \del \delb f =  \delb^* \delb f  - [ \Lambda, \delb^*] L \delb f.
\]
\end{lem}

\begin{cor}
If $(M,J)$ is a compact and connected almost complex manifold then
$
H_{\Dol}^{0,0}(M) \cong \CC.
$
\end{cor}
\begin{proof}
We first note that $H_{\Dol}^{0,0}(M,J) \cong \Ke(\delb) \cap \cA^{0,0}$, so it suffices to show if $\delb f = 0$ then $f$ is constant. For any such $f$ we have $df \circ J = i \,df$, and in any coordinate chart $\phi:U  \to \R^{2n}$ containing any maximum point, we pullback $J$ to 
$\phi(U)$ and consider the $J$-preserving map $f \circ \phi^{-1}: \phi(U) \to \CC$. The components of $d$ are natural with respect this $J$-preserving map and we choose compatible metric on $\phi(U)$ to define $\Lambda$ and $\delb^*$. Then by Lemma \ref{01andf}
\[
-i \Lambda \del \delb f =  \delb^* \delb f  - [ \Lambda, \delb^*] L \delb f
\]
 on $\phi(U)$. Note $\delb^* \delb$ is quadratic, self-adjoint, and positive, and 
$[ \Lambda, \delb^*] L \delb$ is first order since $[ \Lambda, \delb^*]  = [d,L]^*$ is zeroth order, because $[d,L] \eta = d \omega \wedge \eta$. Then $\delb f = 0$ implies the right hand side is zero, so the maximum principle due to E. Hopf applies \cite{Hopf}, showing $f$ is constant in a neighborhood of the maximum point and therefore, by connectedness, $f$ is constant.
\end{proof}

We next show  that the top and bottom rows of Dolbeault cohomology agree with $\delb$-$\mub$-harmonic forms and hence are finite-dimensional. We'll make use of the fact that  $\Delta_\delb+\Delta_\mub$ is elliptic, having the same symbol as $\Delta_\delb$, since $\Delta_\mub$ is linear over functions. In particular, there is a Hodge decomposition theorem for this operator, analogous to Theorem \ref{delbHodge}.

\begin{prop} \label{bottom_top}
Let $\Mh$ be a compact almost Hermitian manifold of dimension $2m$. 
For all 
$0\leq p\leq m$ and $q\in \{0,m\}$ we have 
\[\Hh_{\delb}^{p,q}\cap\Hh_{\mub}^{p,q}\cong H_{\Dol}^{p,q}(M).\]
In particular, for all such $p,q$, $H_{\Dol}^{p,q}(M)$ is finite-dimensional and we have
\[
\dim \left(  \Hh_{\delb}^{p,q}\cap\Hh_\mub^{p,q} \right) =  \dim H_{\Dol}^{p,q}(M) \leq
\dim \Hh_{\delb}^{p,q}.
\]
\end{prop}

\begin{proof}
In bidegrees $(p,0)$, for all 
$0\leq p\leq m$, we have $\mub^*= \delb^* = \delb_\mub^* = 0$, and there is no image of $\mub$ in degrees $(p,0)$ or $(p,1)$, or image of $\delb$ in degrees $(p,0)$. So all three spaces are equal to $\Ke(\delb) \cap \Ke(\mub)$. 

For the case $q=m$, by Lemma \ref{HIT} it suffices to prove that the inclusion
\[
\Hh_{\delb}^{p,m} \cap \Hh_{\mub}^{p,m}\hookrightarrow H_{\Dol}^{p,m}(M)\cong {\cA^{p,m}\over \Img(\mub)+\Img(\delb)}
\]
is an isomorphism.
Let $\omega$ be a representative of $[\omega]\in H_{\Dol}^{p,m}(M)$. Since $\Delta_\delb+\Delta_\mub$ is elliptic, we may write 
\[\omega=\Hh'(\omega)+(\Delta_\delb+\Delta_\mub) G' \omega\]
where $\Hh'$ denotes the projection into $\delb$-$\mub$-harmonic forms and $G'$ denotes the corresponding Green's operator.
For bidegree reasons, we obtain 
\[\omega=\Hh'(\omega)+(\delb\delb^*+\mub\mub^*) G' \omega\]
which shows that $[\omega]=[\Hh'(\omega)]$ and hence the above map is surjective.
\end{proof}

The the fact that $\dim \Hh_{\delb}^{p,0}$ is metric independent for all $p$ was likely observed by Hirzebruch and others, and has been recently generalized for bundle-valued forms in \cite{CZ}. Note as well that the identities of Proposition \ref{bottom_top} are not satisfied in general for other bidegrees.

The following Serre duality is immediate from Proposition \ref{bottom_top} and Proposition \ref{SerreDuality}.

\begin{cor}
Let $M$ be a compact almost complex manifold of dimension $2m$. 
Then for all  $0< p\leq m$,
\[
  \dim H_{\Dol}^{p,0}(M) = \dim  H_{\Dol}^{m-p,m}(M).
\]
\end{cor}

\subsection{Locally constant rank case} \label{LCR}
In this subsection, we restrict to almost complex structures whose operator $\mub$ has locally constant rank. These include
Lie groups and nilmanifolds with left-invariant almost complex structures, as well as nearly K\"ahler manifolds.

\begin{defn} \label{defnLCR}
Let $(M,J)$ be an almost complex manifold. We say $\mub$ has \textit{locally constant rank} if for each 
$(p,q)$, 
\[\mub: \cA_x^{p,q} \to \cA_x^{p-1,q+2}
\]
has constant rank as a function of $x \in M$. 
\end{defn}

In the locally constant rank case, the pointwise Hodge decomposition of Lemma \ref{mubHodge} becomes global:
\begin{thm}[$\mub$-Hodge decomposition]\label{mubHodgeGLOBAL}
Let $\Mh$ 
be an almost Hermitian manifold and assume that $\mub$ has locally constant rank.
For every bidegree $(p,q)$ the following hold:
\begin{enumerate}
 \item There are direct sum decompositions
\[\cA^{p,q}
= \mub (\cA^{p+1,q-2}) \oplus  \Hh_{\mub}^{p,q} \oplus \mub^*(\cA^{p-1,q+2}).\]
\item Every element $\omega\in \cA^{p,q}$ can be written uniquely as 
\[\omega=\Hh_{\mub}(\omega)+\Delta_\mub G_\mub (\omega)\]
where $\Hh_{\mub}(\omega)$ denotes the projection of $\omega$ to the space $\Hh_{\mub}^{p,q}$ of $\mub$-harmonics in 
$\cA^{p,q}$, and $G_\mub$ denotes the corresponding Green's operator.
\item  The map $\omega\mapsto [\omega]$ defines an isomorphism
\[\Hh_{\mub}^{p,q}\cong H_{\mub}^{p,q}(M):={{\Ke(\mub:\cA^{p,q}\longrightarrow \cA^{p-1,q+2})}\over{\mathrm{\Img}(\mub:\cA^{p+1,q-2}\longrightarrow\cA^{p,q})}}.\]
In particular,
every element in $H_{\mub}^{p,q}(M)$ has a unique $\mub$-harmonic representative.
\end{enumerate}
\end{thm}
\begin{proof}
 If $\mub$ has locally constant rank, then the fibers of $\Imt \mub$ and $\Imt \mub^*$ 
have constant rank (over all points $x \in M$) and therefore define vector bundles. Similarly, the fibers of 
$\Ke \Delta_\mub$ define a vector bundle since they are the pointwise complement of $\Imt \mub + \Imt \mub^*$. 
Then the total bundle $\sqcup_x \cA^{p,q}_x$ splits as a direct sum of subbundles, and on global sections we have
 $\cA^{p,q} = \mub (\cA^{p+1,q-2}) \oplus  \Hh_{\mub}^{p,q} \oplus \mub^*(\cA^{p-1,q+2})$ by Lemma \ref{mubHodge}. The remaining claims follow.
\end{proof}

In the locally constant rank case, the Harmonic Inclusion of Lemma \ref{HIT} admits an intermediate space that is 
conceptually and computationally important, and also satisfies Serre duality, that we now introduce.

Consider the operator
\[\delb_\mub:\Hh^{p,q}_\mub\longrightarrow \Hh^{p,q+1}_\mub\text \quad {by} \quad \delb_\mub(\omega):=\Hh_\mub(\delb \omega).\]
By Theorem \ref{mubHodgeGLOBAL} this is well defined when $\mub$ has locally constant rank since the harmonic component 
$\Hh_\mub(\delb \omega)$ is a smooth form.
We first show that this operator squares to zero and its cohomology is just the Dolbeault cohomology.
\begin{thm}\label{HDolidentifications}
Let $(M,J, \langle \, , \, \rangle)$ be an almost Hermitian manifold such that $\mub$ has locally constant rank. 
For every bidegree $(p,q)$, the following are satisfied:
\begin{enumerate}
 \item There is an isomorphism
\[
{\Ke(\delb_\mub:\Hh_\mub^{p,q}\to \Hh_\mub^{p,q+1})
\over 
\Img(\delb_\mub:\Hh_\mub^{p,-1,q}\to \Hh_\mub^{p,q})
}
\stackrel{\cong}{\longrightarrow} 
 H_{\Dol}^{p,q}(M)
 \]  
induced by the identity on representatives.
\item The following hold:
\[\Ke(\delb_\mub)\cap\cA^{p,q} = \left\{ \omega \in \Hh_\mub^{p,q} ; \delb \omega \in \Img(\mub)\right\}\]
and
\begin{eqnarray*}
\Img(\delb_\mub)\cap\cA^{p,q} &\cong & \left\{\omega \in \Hh_\mub^{p,q}; \omega=\mub \alpha+\delb \beta \text{ with }\mub \beta=0\right\}.
\end{eqnarray*}
\end{enumerate}
\end{thm}

\begin{proof}
The following diagram commutes by definition
\[
\xymatrix
{
\cA \ar@{->>}[r]^\pi & \Hh_\mub & H_\mub(M)  \ar[l]_{\cong} \\
\cA \ar[u]^\delb & \Hh_\mub \ar[u]_{\del_\mub}  \ar[r]^{\cong} \ar@{_{(}->}[l]  & H_\mub(M) \ar[u]_{\delb}
}
\]
Here the horizontal isomorphisms are from Theorem \ref{mubHodgeGLOBAL} and $\delb: H_\mub(M) \to H_\mub(M)$ is the induced map on $\mub$-cohomology, which follows since $\delb \mub + \mub \delb = 0$. Note that $\delb: H_\mub(M) \to H_\mub(M)$ squares to zero by the relations in Equations (\ref{eq;drelations}) of section \ref{SecPreliminars}. It follows that  $\delb_\mub\circ\delb_\mub=0$, and so the first assertion holds. The identifications for $\Ke(\delb_\mub)$ and $\Img(\delb_\mub)$ follow from 
 the commuting left square and Theorem \ref{mubHodgeGLOBAL}.
\end{proof}

Similarly, in the locally constant rank case, we can consider the operator
\[\delb_\mub^*:\Hh^{p,q}_\mub\longrightarrow \Hh^{p,q-1}_\mub\quad \text{ defined by }\quad \delb_\mub^*(\omega):=\Hh_\mub(\delb^* \omega).\]
For closed manifolds, this operator is the formal adjoint to $\delb_\mub: \Hh^{p,q}_\mub\longrightarrow \Hh^{p,q+1}_\mub$.  Indeed, by the orthogonality of Theorem \ref{mubHodgeGLOBAL}, the projection operator $\Hh_\mub$ is self-adjoint, so for all $\omega \in \Hh^{p,q}_\mub$ and $\eta \in \Hh^{p,q-1}_\mub$ we have
\[
\langle \Hh_\mub(\delb^* \omega) , \eta \rangle
=
\langle \omega , \delb \eta \rangle
=
\langle \omega , \Hh_\mub( \delb \eta) \rangle
=
\langle \omega , \delb_\mub \eta \rangle.
\]
Note that $\delb^*_\mub$ anti-commutes with $\mub^*$ and therefore 
\[
\Ke \delb_\mub^* \cap  \cA^{p,q} = \left\{x\in \Hh_\mub^{p,q} ; \delb^* x\in \Img(\mub^*)\right\}.
\]

We define the \textit{$\delb_\mub$-Laplacian} on $\Hh_\mub$ by  \[\Delta_{\delb_\mub}:=\delb_\mub\delb_\mub^*+\delb_\mub^*\delb_\mub.\]

\begin{defn}
The space of \emph{$\delb_\mub$-harmonic forms} of an almost Hermitian manifold is given by
\[\Hh_{\delb_\mub}^{p,q}:= \Ke(\Delta_{\delb_\mub})  = \Ke(\mub)\cap \Ke(\mub^*)\cap \Ke(\delb_\mub)\cap \Ke(\delb_\mub^*) \cap \cA^{p,q}.\]
\end{defn}

The $\delb_\mub$-harmonic forms, which a priori are infinite-dimensional, also satisfy Serre duality:

\begin{prop}[Serre duality]\label{HhdelbmubSerre}
Let $(M,J, \langle \, , \, \rangle)$ be an almost Hermitian manifold of dimension $2m$ such that $\mub$ has locally constant rank. 
For every bidegree $(p,q)$  we have isomorphisms

\[
\Hh_{\mub}^{p,q} \cong \Hh_{\mub}^{m-p,m-q}
\text{ and }
\Hh_{\delb_\mub}^{p,q} \cong \Hh_{\delb_\mub}^{m-p,m-q}.\]
\end{prop}

\begin{proof} The first claim follows from Theorem \ref{mubHodgeGLOBAL} and Proposition \ref{SerreDuality}.
For the second claim, it suffices to show that $\bar \star$ commutes with $\Delta_{\delb_\mub}$, where $\bar \star$ is the $\star$-operator followed by conjugation. This will follow immediately from the relations 
\[\bar \star \delb_\mub^* =  (-1)^k \delb_\mub \bar \star\text{ and } \delb_\mub^* \bar \star = (-1)^{k+1}  \bar \star \delb_\mub,\] which we show hold on $\cA^k$. Recall that $\star^2 = (-1)^k$ on $\cA^k$.
Since $\delb^* = - \star \del \star$, it follows that $\bar \star \delb^* = (-1)^k \delb \bar \star$ on $\cA^k$,
and $\del^* = - \star \delb \star$ implies $\delb^* \bar \star  = (-1)^{k+1} \bar \star \delb $ on $\cA^k$. 
Similarly,   $\bar \star \mub^* = (-1)^k\mub \bar \star$ and $\mub^* \bar \star  = (-1)^{k+1}  \bar \star \mub$ on $\cA^k$.
Therefore, $\bar \star$ respects the orthogonal $\mub$-Hodge decomposition of Theorem \ref{mubHodgeGLOBAL}, and so $\bar \star$ also commutes with the orthogonal projection operator $\Hh_\mub$ onto $\mub$-harmonics. The result follows since by definition $\delb_\mub = \Hh_\mub \circ \delb$ and $\delb_\mub^* = \Hh_\mub \circ \delb^*$.
\end{proof}

Finally,  in the case that $\mub$ has locally constant rank, there is an improved version of the Harmonic Inclusion of Lemma \ref{HIT}.
This Theorem will be used in section \ref{NKSection} to calculate Dolbeault cohomology groups for nearly K\"ahler $6$-manifolds.

\begin{thm}[Harmonic Inclusion; Locally Constant Rank] 
\label{HITCR}
Let $(M,J, \langle \, , \, \rangle)$ be an almost Hermitian manifold such that $\mub$ has locally constant rank. 
For every bidegree $(p,q)$ we have 
\[
\Hh_\delb^{p,q} \cap \Hh_{\mub}^{p,q} \,  \subseteq  \, \Hh_{\delb_\mub}^{p,q}  \, \stackrel{\star}{\hookrightarrow} \, 
{\Ke(\delb_\mub)\over\Img(\delb_\mub)}\cong
H_{\Dol}^{p,q}(M).
\]
The injection $(\star)$ is an isomorphism if $\Img(\delb_\mub)^\perp = \Cok(\delb_\mub)$.
\end{thm}

\begin{proof}
The first containment follows since 
\begin{align*}
\Ke(\delb) \cap \Hh_{\mub}^{p,q} \,  &\subseteq  \, \Ke( \delb_\mub) \\ 
\Ke(\delb^*) \cap \Hh_{\mub}^{p,q} \,  &\subseteq  \, \Ke( \delb_\mub^*) .
\end{align*}The second follows since
\[
 \Ke(\Delta_{\delb_\mub})  =  \Ke(\delb_\mub) \cap \Ke(\delb_\mub^*) = \Ke(\delb_\mub) \cap \Img \left(\delb_\mub\right)^\perp,
\]
and the fact that for any linear map $\varphi$ of inner produce spaces, $\Img(\varphi)^\perp$ injects into $\Cok(\varphi)$. 
\end{proof}

\begin{rmk}
A natural question is whether in the compact case there are isomorphisms
\[\Hh_{\delb_\mub}^{p,q}\cong H_{\Dol}^{p,q}(M),\] which would 
show the spaces on the left-hand side are metric-independent.
Such an isomorphism would follow immediately from having a Hodge decomposition theorem with respect to the operator $\delb_\mub$. However, the operator $\delb_\mub$ is not elliptic, see Remark \ref{delbmubnonelliptic} for an example. Nevertheless, we will see that in the case of left-invariant almost complex structures on compact Lie groups and nilmanifolds, this isomorphism always holds.
\end{rmk}

We include one more lemma here concerning bidegree $(0,1)$, which will be used in Examples \ref{HDolMNI4} and \ref{HDolMNI6}. 

\begin{lem}\label{H01}
Let $\Mh$ be a compact almost Hermitian manifold such that $\mub$ has locally constant rank. 
Then \[H_{\Dol}^{0,1}(M)\cong \Hh_{\delb_\mub}^{0,1}.\]
\end{lem}
\begin{proof}
Let $\omega$ be a representative of an element in $H^{0,1}_{\Dol}\cong \Ke(\delb_\mub)/\Img(\delb)$. 
 Then $\omega':=\omega-\delb G_\delb \delb^* \omega$ is another representative, and 
 \[
 \delb^* \omega =  \Hh_\delb(\delb^*\omega) +  \Delta_\delb G_\delb \delb^* \omega   =   \delb^*  \delb G_\delb \delb^* \omega 
 \]
 since $\Hh_\delb(\delb^*\omega)=0$, by Theorem \ref{delbHodge}. Therefore, $\delb^* \omega'=0$, so $\omega' \in \Hh_{\delb_\mub}^{0,1}$.
\end{proof}

\begin{rmk}
 The above result shows that the space $\Hh_{\delb_\mub}^{0,1}$ is always metric-independent. 
 Note that, in contrast, and as shown by Holt and Zhang in \cite{HZ}, for some almost complex manifolds the classical space of $\delb$-harmonic forms
 $\Hh_\delb^{0,1}$ may depend on the metric.
\end{rmk}

\section{Lie groups and nilmanifolds}\label{SecLie}
A particularly useful set of examples is given by Lie algebra cohomology, which allows one to compute purely
geometric invariants for compact Lie groups and nilmanifolds, after solving finite-dimensional linear algebra problems.
In this section, we study the Fr\"{o}licher spectral sequence in the context of Lie algebras with arbitrary almost complex structures, and develop a harmonic theory for their Dolbeault cohomology in the non-integrable case.
We then translate the results on Lie-algebra Dolbeault cohomology to the geometric situation of compact Lie groups and nilmanifolds respectively, and deduce several results in that setting. An extension of these results to the case of solvmanifolds is presented in \cite{SiTo}.

\subsection{Dolbeault cohomology of Lie algebras}
Let $\mathfrak{g}$ be real 
Lie algebra of finite dimension $2m$ and let $J$ be an almost complex structure on its underlying real vector space, i.e.
an endomorphism $J:\mathfrak{g}\to \mathfrak{g}$ such that $J^2=-1$. This determines decompositions 
\[\mathfrak{g}_\CC=\mathfrak{g_\CC}^{1,0}\oplus\mathfrak{g_\CC}^{0,1}\text{ and }
\mathfrak{g}_\CC^\vee=(\mathfrak{g^\vee_\CC})^{1,0}\oplus(\mathfrak{g^\vee_\CC})^{0,1}
\]
on the complexified Lie algebra $\mathfrak{g}_\CC:=\mathfrak{g}\otimes_{\R}\CC$,
and on its dual $\mathfrak{g}_\CC^\vee$.
The exterior algebra 
\[\cA^*_{\mathfrak{g}_\CC}:=\bigoplus_{k\geq 0} \Lambda^k(\mathfrak{g}^\vee_\CC)\]
then becomes a bigraded algebra $\cA^*_{\mathfrak{g}_\CC}=\bigoplus \cA^{p,q}_{\mathfrak{g}}$ with
\[
\cA^{p,q}_{\mathfrak{g}}:=\Lambda^p\left((\mathfrak{g^\vee_\CC})^{1,0}\right) \wedge \Lambda^q \left((\mathfrak{g^\vee_\CC})^{0,1} \right).
\]
Recall that the differential 
on $\mathfrak{g}_\CC^\vee$ is defined as the negative of the dual of the Lie bracket
\[d|_{\mathfrak{g}^\vee_\CC}:=[-,-]^\vee:\mathfrak{g}_\CC^\vee\to \mathfrak{g}_\CC^\vee\wedge \mathfrak{g}_\CC^\vee,\]
and is extended uniquely to a derivation $d$ of $\cA^*_{\mathfrak{g}_\CC}$.
It follows that $d$ decomposes into four components
\[d=\mub +\delb +\del +\mu\]
with $|\mub|=(-1,2)$, $|\delb|=(0,1)$, etc., and these satisfy Equations (\ref{eq;drelations}) of section \ref{SecPreliminars}.
In particular, the pair $(\mathfrak{g},J)$ has a well-defined a $\mub$-cohomology, denoted by $H_\mub^{p,*}(\mathfrak{g},J)$, and a well-defined Hodge filtration $F$ as in Definition \ref{Hodgefildef}. This gives an obvious notion of \textit{Dolbeault cohomology}
\[H_{\Dol}^{p,q}(\mathfrak{g},J):=H^{q}(H_\mub^{p,*}(\mathfrak{g},J),\delb)\]
and the same proof of Theorem \ref{E1HDol} gives isomorphisms 
\[H_{\Dol}^{*,*}(\mathfrak{g},J)\cong E_1^{*,*}(\cA^*_{\mathfrak{g}_\CC},F)
\cong
{\left\{ \omega \in \cA^{p,q}_{\mathfrak{g}}\cap\Ke(\mub); \delb  \omega  \in\Img(\mub)\right\}
\over{
\left\{  \omega  \in \cA^{p,q}_{\mathfrak{g}}\text{ such that }  \omega  =\mub \alpha +\delb \beta \text{ with } \mub \beta =0\right\}
}}.
\]
A main advantage of this toy-model for the Fr\"{o}licher spectral sequence
is that
the computation of the spaces $E_r^{*,*}(\mathfrak{g},F)$ 
is a finite-dimensional linear algebra problem.
A first consequence is the following result.
Denote by 
\[
h^{p,q}(\mathfrak{g},J):=\dim H_{\Dol}^{p,q}(\mathfrak{g},J) \text{ and }
b^{n}(\mathfrak{g}):= \dim H^n(\cA^*_{\mathfrak{g}_\CC},d)
\]
the Dolbeault and complex Betti numbers of $(\mathfrak{g},J)$. Then:
\begin{prop}\label{FrolicherCoro}
Let $\mathfrak{g}$ be real 
Lie algebra of finite dimension $2m$ and $J$ an almost complex structure on its underlying real vector space. 
\begin{enumerate}
 \item For all $n\geq 0$ we have inequalities 
 \[\sum_{p+q=n} h^{p,q}(\mathfrak{g},J)\geq b^n(\mathfrak{g}).\]
 \item Let $\chi(\mathfrak{g}):=\sum (-1)^n b^n(\mathfrak{g})$ denote the Euler characteristic of $\mathfrak{g}$. Then
 \[\chi(\mathfrak{g})=\sum_{p,q=0}^{m}(-1)^{p+q}\,h^{p,q}(\mathfrak{g},J).\]
\end{enumerate}
\end{prop}
\begin{proof}
This is a classical consequence of the convergence of the above spectral sequence, 
together with the fact that all vector spaces are finite-dimensional.
\end{proof}

We next develop a harmonic theory for the Dolbeault cohomology of $(\mathfrak{g},J)$.
 We'll be rather brief on some details since the background and theory somewhat mimics the smooth manifold setting from Section \ref{SecHarm}.

Let $\langle-,-\rangle$ be a real inner product on $\mathfrak{g}$ such that $\langle Jx,J y \rangle=\langle x, y \rangle$, which may be extended to a Hermitian inner product on $\mathfrak{g}_\CC$.
Define the Hodge star operator  \[\star:\cA^{p,q}_{\mathfrak{g}}\longrightarrow \cA^{m-p,m-q}_{\mathfrak{g}}
\quad\text{by}\quad
\omega\wedge \star \ov \eta \ :=\langle \omega, \eta \rangle\Omega,\]
where $\Omega$ is the volume element determined by $J$ and the metric.

Let $\delta$ denote one of the components $\mub$, $\delb$, $\del$ or $\mu$ of $d$.
We then define
\[\delta^*:=- \star \delta \star\text{ and }
\Delta_\delta:=\delta\delta^*+\delta^*\delta.\]
We then have $\star \Delta_{\bar \delta} = \Delta_\delta \star$ for $\delta = \delb, \mub$ and define the space of 
\textit{$\delta$-harmonic forms} by
\[\Hh_\delta^{p,q}({\mathfrak{g}}):=\Ke(\Delta_\delta) \cap \cA^{p,q}_{\mathfrak{g}}.\]
As in Lemma \ref{mubHodge}, we have a $\mub$-Hodge decomposition
\[\cA^{p,q}_{\mathfrak{g}}= \mub\left( \cA^{p+1,q-2}_{\mathfrak{g}} \right)  \oplus \Hh_\mub^{p,q}({\mathfrak{g}})\oplus \mub^* \left( \cA^{p-1,q+2}_{\mathfrak{g}} \right) \]
into orthogonal spaces. On $\Hh_\mub^{*,*}(\mathfrak{g})$ define operators 
\[\delb_\mub:\Hh_\mub^{p,q}(\mathfrak{g})\to \Hh_\mub^{p,q+1}(\mathfrak{g})\text{ and }\delb_\mub^*:\Hh_\mub^{p,q}(\mathfrak{g})\to \Hh_\mub^{p,q-1}(\mathfrak{g})\]
by letting 
\[\delb_\mub(\alpha):=\Hh_\mub(\delb \alpha)\text{ and }\delb_\mub^*:=\Hh_\mub(\delb^* \alpha)\]
where $\delb^*:=-\star \del \star$ and $\Hh_\mub(\alpha)$ denotes the projection of $\alpha$ into the space of $\mub$-harmonic forms. As in Theorem \ref{HDolidentifications} we have
\[
{\Ke(\delb_\mub:\Hh_\mub^{p,q}(\mathfrak{g})\to \Hh_\mub^{p,q+1}(\mathfrak{g}))
\over 
\Img(\delb_\mub:\Hh_\mub^{p-1,q}(\mathfrak{g})\to \Hh_\mub^{p,q})(\mathfrak{g})
}
\stackrel{\cong}{\longrightarrow} 
 H_{\Dol}^{p,q}(\mathfrak{g},J)
 \] 
 
 Define the space of \textit{$\delb_\mub$-harmonic forms} as
\[\Hh_{\delb_\mub}^{p,q}(\mathfrak{g}):=\Ke(\Delta_{\delb_\mub})\cap \Hh_\mub^{p,q}(\mathfrak{g})=\Ke(\delb_\mub)\cap\Ke(\delb_\mub^*)\cap \Hh_\mub^{p,q}(\mathfrak{g}).\]

The following generalizes a result of \cite{Rollenske2} to the non-integrable setting.

\begin{lem}\label{adjointLie}
Assume that $H^{2m}(\mathfrak{g}_\CC)\cong\CC$. Then $\delb_\mub^*$ is adjoint to $\delb_\mub$.
\end{lem}

\begin{proof}
The condition $H^{2m}(\mathfrak{g}_\CC)\cong \CC$, the algebraic analogue of having a closed manifold, implies that $d$ vanishes on $\cA^{2m-1}_{\mathfrak{g}_\CC}$,  so that
$\del$ vanishes on $\cA^{m-1,m}_{\mathfrak{g}}$ as well. Then the usual argument, as in the smooth global setting, shows that $\delb^*$ is the adjoint of $\delb$ (c.f proof of Lemma \ref{mub*adjoint}). It follows that $\delb_\mub^*$ is adjoint to $\delb_\mub$, since the projection onto $\mub$-harmonics is self adjoint: for all $\omega \in \Hh^{p,q}_\mub(\mathfrak{g})$ and $\eta \in \Hh^{p,q-1}_\mub(\mathfrak{g})$, 
\[
\langle \Hh_\mub(\delb^* \omega) , \eta \rangle
=
\langle \omega , \delb \eta \rangle
=
\langle \omega , \Hh_\mub( \delb \eta) \rangle
=
\langle \omega , \delb_\mub \eta \rangle. \qedhere\]
\end{proof}

\begin{rmk}[c.f.\cite{Rollenske2}] \label{compctandnil}
 Note that the condition of the Lemma is satisfied when $\mathfrak{g}$ is the Lie algebra of a connected compact Lie group, and also for a nilpotent Lie group.
 Indeed, for any $\mathfrak{g}$ the top cohomology with values in the module $\cA^{2m}_{\mathfrak{g}_\CC}$  satisfies  $H^{2m}(\mathfrak{g}_\CC,\cA^{2m}_{\mathfrak{g}_\CC})\cong \CC$.
 Hence the condition $H^{2m}(\mathfrak{g}_\CC)\cong\CC$ is satisfied whenever $\mathfrak{g}$ acts trivially on 
 $\cA^{2m}(\mathfrak{g})$. This is the case for $\mathfrak{g}$ nilpotent.
 Lemma \ref{adjointLie} also applies to unimodular Lie algebras, for which $d\equiv 0$ on $\cA^{2m-1}_{\mathfrak{g}_\CC}$ (see for instance \cite{SiTo}).

\end{rmk}

\begin{thm} \label{LiedelbmubHodge}
Assume that $H^{2m}(\mathfrak{g}_\CC) \cong \CC$. Then for any almost complex structure $J$ on $\mathfrak g$, and all $(p,q)$, there is a $\delb_\mub$-Hodge decomposition
\[
\Hh_{\mub}^{p,q}(\mathfrak{g})= \delb_\mub\left( \Hh_{\mub}^{p,q-1}(\mathfrak{g})  \right) \oplus \Hh_{\delb_\mub}^{p,q}(\mathfrak{g})\oplus  \delb_\mub^* \left(  \Hh_{\mub}^{p,q+1}(\mathfrak{g}) \right) 
\]
and there are isomorphisms
\[
H_{\Dol}^{p,q}(\mathfrak{g},J)\cong\Hh_{\delb_\mub}^{p,q}(\mathfrak{g}).
\]
\end{thm}
\begin{proof}This follows from linear algebra since $\delb_\mub^2=0$, as in the proof of Lemma \ref{mubHodge}.
\end{proof}

By an argument identical to the proof of Proposition \ref{HhdelbmubSerre}, we find that
\[
\bar \star: \Hh_{\delb_\mub}^{p,q}(\mathfrak{g}) \to \Hh_{\delb_\mub}^{m-p,m-q}(\mathfrak{g})
\]
is an isomorphism, where $\bar \star$ is the $\star$-operator followed by conjugation. 
So, by Theorem \ref{LiedelbmubHodge}, we obtain Serre duality on the Dolbeault cohomology:
\begin{cor}[Serre duality]
Assume that $H^{2m}(\mathfrak{g}_\CC) \cong \CC$. For any almost complex structure $J$ on $\mathfrak g$ there are isomorphisms
\[H_{\Dol}^{p,q}(\mathfrak{g},J)\cong H_{\Dol}^{m-p,m-q}(\mathfrak{g},J).\]
\end{cor}

\begin{rmk}
The above Serre duality result allows to reduce the number of terms in the formulas of Proposition \ref{FrolicherCoro},
in the case when $\mathfrak{g}$ is the Lie algebra of a connected compact or nilpotent Lie group.
In addition, by work of Milivojevi\'{c} \cite{MiliSerre}, Serre duality persists in every stage of the 
Fr\"{o}licher spectral sequence.
\end{rmk}

\subsection*{Compact Lie groups}
If $G$ is any Lie group with Lie algebra $\mathfrak{g}$, then the algebra $\cA^*_{\mathfrak{g}}$ is isomorphic to
the algebra $\cA_L^*(G)$ of left-invariant forms on $G$ and the  almost complex structure on $\mathfrak{g}$ defines a unique left invariant almost complex structure on $G$  (with every left invariant almost complex structure on $G$ occuring uniquely in this way). Moreover, the isomorphism
\[\cA^*_{\mathfrak{g}_\CC}\cong \cA^{*}_L(G)\otimes\CC\]
preserves the bigradings.
Note that by Lemma \ref{lem;Nmub}, the integrability of $J$ on $G$ is equivalent to $\mub\equiv 0$,
which occurs if and only if 
 $\mub\equiv 0$ in $\cA^{1,0}$.
\begin{defn}
We define the \textit{left-invariant Dolbeault cohomology} of $(G,J)$
by 
\[^LH_{\Dol}^{p,q}(G,J):=H_{\Dol}^{p,q}(\mathfrak{g},J).\]
The \textit{left-invariant Fr\"{o}licher spectral sequence} of $(G,J)$ is the spectral sequence
\[^LE_1^{p,q}(G,J):=E_1^{p,q}(\cA^*_{\mathfrak{g}_\CC},F)\Longrightarrow {}^LH^{p+q}(G,\CC).\]
\end{defn}

Denote by 
\begin{align*}
h^{p,q}_L(G,J):=\dim {}^LH_{\Dol}^{p,q}(G,J)\quad\text{ and }\quad
b^{n}_L(G):= {}^LH^n(G,\CC)
\end{align*}
the left-invariant Dolbeault and Betti numbers of $(G,J)$. Proposition \ref{FrolicherCoro} gives:
\begin{cor}\label{FrolicherCoroG}
Let $G$ be a real Lie group of dimension $2m$ with a left invariant almost complex structure $J$. 
\begin{enumerate}
 \item For all $n\geq 0$ we have inequalities 
 \[\sum_{p+q=n} h^{p,q}_L(G,J)\geq b^n_L(G).\]
 \item Let $\chi_L(G):=\sum (-1)^n b^n_L(G)$ denote the left-invariant Euler characteristic of $G$. Then
 \[\chi_L(G)=\sum_{p,q=0}^{m}(-1)^{p+q}\,h^{p,q}_L(G,J).\]
\end{enumerate}
\end{cor}

For the remaining of this subsection, assume that $G$ is a compact Lie group of dimension $2m$
with a left-invariant almost complex structure $J$.
Then the inclusion
\[\cA^*_{\mathfrak{g}}\cong \cA_L^*(G)\hookrightarrow \cA^*(G)\]
is known to be a quasi-isomorphism, so we have 
\[{}^LH^{*}(G,\CC)\cong H^{*}(G,\CC).\]
In particular, the
left-invariant Fr\"{o}licher spectral sequence of $G$ computes its complex cohomology
and Corollary \ref{FrolicherCoroG} applies with $b^n_L(G)=b^n(G)$ and $\chi_L(G)=\chi(G)$.

Furthermore, since $(G,J)$ is an almost complex manifold, it has an associated (non-left-invariant) Dolbeault cohomology
and Fr\"{o}licher spectral sequence. We have:

\begin{lem}\label{injGDol}
Let $G$ be a compact Lie group with a left-invariant almost complex structure $J$. For all $r\geq 0$, the inclusion $\cA_L(G) \hookrightarrow \cA(G)$ induces an injection
\[{}^LE_{r}^{*,*}(G,J)\longrightarrow  E_{r}^{*,*}(G,J),\]
which becomes an isomorphism at the $E_\infty$-stages.
\end{lem}
\begin{proof}
Let 
\[
I(\omega) = \int_G L_g^* \omega
\] 
 be the averaging operator with respect to the normalized Haar measure, where $L_g$ is the left translation on $G$.
 Since the almost complex structure is preserved by $L_g$, $I$ commutes with the projections onto the $(p,q)$ spaces, and therefore $I$ commutes with the components of $d$.

Since the inclusion of left-invariant forms into all forms preserves bigradings, we have well-defined maps of spectral sequences 
 \[{}^LE_r^{*,*}(G,J)\longrightarrow E_r^{*,*}(G,J).\]
Therefore the isomorphism at the $E_\infty$-stages
follows from the fact that both spectral sequences converge to $H^*(G,\CC)$.
 
To prove that the above maps are injective for all $r\geq 1$,
we use the formulas for $E_r=Z_r/B_r$ given in the Appendix.
Let $\omega$ be left-invariant and assume that $\omega\in B_r(G,J)$. Since $\omega=I(\omega)$ and $I$ commutes with the components of 
$d$ one may check, using the description of $B_r$, 
that $\omega=I(\omega)\in {}^LB_r(G,J)$. Therefore $[\omega]$ is zero in ${}^LE_r(G,J)$.
\end{proof}

Note that the case $r=1$ above gives an injection of Dolbeault cohomologies
\[{}^LH_{\Dol}^{p,q}(G,J)\hookrightarrow H_{\Dol}^{p,q}(G,J).\]
This injection admits an alternative proof via the theory of harmonic forms as we will next see.
Denote by
\[\cA^*_L(G)\otimes\CC=\bigoplus \cA^{p,q}_L\]
the bigrading induced by $J$ on the algebra of left-invariant forms of $G$.
A left-invariant almost Hermitian metric on $G$ allows to define spaces
\[{}^L\Hh_{\delb_\mub}^{p,q}:=\Ke(\Delta_{\delb_\mub})\cap \cA^{p,q}_L \cong \Hh_{\delb_\mub}^{p,q}(\mathfrak{g})\]
of \textit{left-invariant $\delb_\mub$-harmonic forms}.

By Lemma \ref{HIT}, Theorem \ref{LiedelbmubHodge}, and its consequences, we obtain:
\begin{cor}\label{mainGcompact}
Let $G$ be a compact Lie group with a left-invariant almost complex
structure $J$. For all $(p,q)$ the following is satisfied:
\begin{enumerate}
 \item There is an inclusion
\[{}^L H_{\Dol}^{p,q}(G,J)\hookrightarrow H_{\Dol}^{p,q}(G,J).\]
\item We have Serre duality isomorphisms  
\[{}^L H_{\Dol}^{p,q}(G,J)\cong {}^L H_{\Dol}^{p,q}(G,J)^{m-p,m-q}.\]
\item For any left-invariant compatible metric on $(G,J)$, we have
a diagram of injective maps
\[
\xymatrix
{
{}^L\left( \Hh_\delb^{p,q} \cap \Hh_{\mub}^{p,q}\right)   \ar@{^{(}->}[r]   \ar@{^{(}->}[d] &  {}^L \Hh_{\delb_\mub}^{p,q}  \ar[r]^-{\cong} \ar@{^{(}->}[d]  &  
 {}^LH_{\Dol}^{p,q}(G,J) \ar@{^{(}->}[d]  \\
\Hh_\delb^{p,q} \cap \Hh_{\mub}^{p,q}  \ar@{^{(}->}[r] &  \Hh_{\delb_\mub}^{p,q}  \ar@{^{(}->}[r] & 
H_{\Dol}^{p,q}(G,J) 
}
\]
and an isomorphism ${}^L\Hh_{\delb_\mub}^{p,q}\cong {}^LH_{\Dol}^{p,q}(G,J)$.
In particular, the numbers $\dim {}^L \Hh_{\delb_\mub}^{p,q}$ are metric-independent.
\end{enumerate}
\end{cor}

\begin{rmk}\label{allinvariants}
The spaces in the top row are all finite-dimensional and straightforward to compute from the Lie algebra of $G$. 
In favorable situations, one can deduce from the top row information concerning the bottom row.
For instance, in the next example below we conclude, via left-invariant computations, that the Dolbeault numbers of two different almost complex structures on the same manifold are distinct.
Note as well that in general, the inclusion ${}^L\left( \Hh_\delb^{p,q} \cap \Hh_{\mub}^{p,q}\right)\hookrightarrow {}^L\Hh_{\delb_\mub}^{p,q}$ is strict, as can be checked in the example below.
\end{rmk}

\begin{ex}\label{su2su2}
Let $G=SU(2)\times SU(2)$ and consider its Lie algebra, generated by
\[\{X_1,Y_1,Z_1,X_2,Y_2,Z_2\}\]
with the only non-trivial brackets given by 
\[[X_i,Y_i]=2Z_i,\,[Y_i,Z_i]=2X_i,\,[Z_i,X_i]=2Y_i\text{ for }i=1,2.\]
Let $J$ be the almost complex structure defined by 
\[J(X_1)=X_2, J(Y_1)=Y_2, J(Z_1)=Z_2, J(X_2)=-X_1, J(Y_2)=-Y_1, J(Z_2)=-Z_1.\]
Define $X:=X_2+iX_1$, $Y:=Y_2+iY_1$ and $Z:=Z_2+iZ_1$.
The non-trivial Lie brackets are given by
\[
[X,Y]=[\ov X,\ov Y]=k Z+\ov k \ov Z\text{ and }[X,\ov Y]=[\ov X,Y]=\ov k Z+ k\ov Z\]
where $k:=(1+i)$, as well as those given by the cyclic permutations of $X,Y,$ and $Z$. 
The dual Lie algebra is given by
\[\Lambda(x,y,z,\ov x,\ov y,\ov z)\]
and its differential is determined by
\[
\arraycolsep=4pt\def\arraystretch{1.4}
\begin{array}{llll}
\mub x=-k \ov y\ov z,&\delb x=-\ov k (y\ov z+\ov y z),& \del x=-k yz,&\mu x=0.\\
\mub y=-k \ov z\ov x,&\delb y=-\ov k (z\ov x+\ov z x),& \del y=-k zx,&\mu y=0.\\
\mub z=-k \ov x\ov y,&\delb z=-\ov k (x\ov y+\ov x y),& \del z=-k xy,&\mu z=0.\\
\end{array}
\]
together with the corresponding complex conjugate equations.
We obtain
\[{}^LH_{\Dol}^{*,*}(G,J)\cong 
\arraycolsep=4pt\def\arraystretch{1.4}
 \begin{array}{|c|c|c|c|c|}
 \hline
0&0&0&\CC\\
\hline
0&\CC&\CC^3&\CC^3
\\
\hline
\CC^3&\CC^3&\CC&0\\
\hline
\CC& 0&0&0\\
\hline
\end{array}\text{ and }
{}^LE_2^{*,*}(G,J)\cong
\arraycolsep=6pt\def\arraystretch{1.4}
 \begin{array}{|c|c|c|c|c|}
 \hline
0&0&0&\CC\\
\hline
0&\CC&0&0\\
\hline
0&0&\CC&0\\
\hline
\CC& 0&0&0\\
\hline
\end{array}
\]

This is an example of a nearly K\"{a}hler $6$-manifold, whose properties are further studied in Section \ref{NKSection}.
A known propery of nearly K\"{a}hler $6$-manifolds is that 
even locally they have no holomorphic functions except
constants \cite{Mus}.
In particular, the sheaf cohomology group $H^1(G,\Omega^0)$ is isomorphic to $H^1_{\dR}(G;\CC)$, which is trivial for $G$.
In contrast, from the table above we have $^{L}H_{\Dol}^{0,1}(G,J)\cong \CC^3$. This exhibits how Dolbeault cohomology can be a rich invariant even when there are no non-constant holomorphic functions.

Note that ${}^LE_2^{*,*}(G,J) \cong H^{*,*}_{dR}(G,J)$. Also, by Corollary \ref{NK6E2} proved below, we have 
$E_2^{*,*}(G,J) = H^{*,*}_{dR}(G,J)$ as well, so we may conclude that $E_2^{*,*}(G,J)= {}^LE_2^{*,*}(G,J)$.

The manifold $SU(2)\times SU(2)$ is diffeomorphic to $S^3\times S^3$, which can be endowed with an
integrable complex structure $J'$ via the holomorphic fibration
\[S^1\times S^1\longrightarrow S^3\times S^3\longrightarrow\CC\PP^1\times\CC\PP^1.\]
Its Fr\"{o}licher spectral sequence is known, given by:

\[H_{\Dol}^{*,*}(G,J')\cong 
\arraycolsep=6pt\def\arraystretch{1.4}
 \begin{array}{|c|c|c|c|c|}
 \hline
0&0&0&\CC\\
\hline
0&\CC&\CC&\CC
\\
\hline
\CC&\CC&\CC&0\\
\hline
\CC& 0&0&0\\
\hline
\end{array}\text{ and }
E_2^{*,*}(G,J')\cong
\arraycolsep=6pt\def\arraystretch{1.4}
 \begin{array}{|c|c|c|c|c|}
 \hline
0&0&0&\CC\\
\hline
0&\CC&0&0\\
\hline
0&0&\CC&0\\
\hline
\CC& 0&0&0\\
\hline
\end{array}
\]
In particular, by Lemma \ref{injGDol} we have
\[3=\dim {}^LH_{\Dol}^{0,1}(G,J)\leq \dim H_{\Dol}^{0,1}(G,J)\neq \dim H_{\Dol}^{0,1}(G,J')=1,\]
so the (non-left invariant) Dolbeault cohomologies of these two almost complex structures are distinct as well.
\end{ex}

\subsection*{Nilmanifolds}
Let $G$ be a nilpotent Lie group with Lie algebra $\mathfrak{g}$. 
If $\mathfrak{g}=\mathfrak{g}_\mathbb{Q}\otimes\mathbb{R}$
has a rational structure, then by \cite{Malcev} there exists a 
discrete subgroup $\Gamma$ such that the quotient $M=\Gamma \backslash G$ is a compact manifold, called a \textit{nilmanifold}.
The algebra $\cA^*_{\mathfrak{g}}$
may be regarded as the algebra of left-invariant forms on $G$.
There is an inclusion $\cA^*_{\mathfrak{g}}\hookrightarrow \cA^*_{\dR}(M)$ and by Nomizu's Theorem \cite{Nomizu}, it induces an isomorphism
\[H^*(\mathfrak{g})\cong H^*(M,\R).\]

A left $G$-invariant $J$ on $G$ induces an almost complex structure on $M$, which is integrable 
if and only if $J$ is integrable on the Lie group $G$ of $\mathfrak{g}$.

The inclusion
\[\cA^*_{\mathfrak{g_\CC}}\hookrightarrow \cA^*_{\dR}(M)\otimes\CC\]
is compatible with the bigradings, but in general, even in the integrable case it is not known if it induces 
an isomorphism on Dolbeault cohomology, although this is the case in several situations (see \cite{CFGU2}, \cite{ConFi}, \cite{FRR})
and conjectured to be true in the integrable case (see \cite{Rollenske}).
Recent results of Coelho, Placini and Stelzig \cite{Stelzigandfriends} imply that, on 4- or 6-dimensional nilmanifolds with left-invariant maximally non-integrable
almost complex structures, Dolbeault cohomology can never be computed using left-invariant
forms only.

\begin{defn}
Define the \textit{left-invariant Dolbeault cohomology} of $(M,J)$
by letting
\[{}^{L} H_{\Dol}^{*,*}(M,J):=H^{*,*}_{\Dol}(\mathfrak{g},J)\] 
Define the \textit{left-invariant
Fr\"{o}licher spectral sequence} for $(M,J)$ by
\[{}^{L}E_1^{*,*}(M,J):=E_1^{*,*}(\cA^*_{\mathfrak{g}_\CC},F).\]
\end{defn}

It follows that ${}^{L} H_{\Dol}^{*,*}(M,J)$ is a well-defined invariant of the almost complex structure of $M$
and the left-invariant Fr\"{o}licher spectral sequence converges to the complex cohomology $H^*(M,\CC)$.

Note that by Remark \ref{compctandnil}, Theorem \ref{LiedelbmubHodge} and its consequences
apply to nilmanifolds,
giving the following.

\begin{cor}
Let $M= \Gamma \backslash G$ be a nilmanifold of dimension $2m$ with a left-invariant almost complex structure $J$. 
Then for all $(p,q)$, the following is satisfied:
\begin{enumerate}
 \item There is an inclusion
\[{}^{L} H_{\Dol}^{p,q}(M,J)\hookrightarrow H_{\Dol}^{p,q}(M,J).\]
\item We have Serre duality isomorphisms  
\[{}^{L} H_{\Dol}^{p,q}(M,J)\cong {}^{L} H_{\Dol}^{p,q}(M,J)^{m-p,m-q}.\]
\item For any left-invariant compatible metric on $(M,J)$, 
we have
a diagram of injective maps
\[
\xymatrix
{
{}^{L}\left( \Hh_\delb^{p,q} \cap \Hh_{\mub}^{p,q}\right)   \ar@{^{(}->}[r]   \ar@{^{(}->}[d] &  {}^{L} \Hh_{\delb_\mub}^{p,q}  \ar[r]^-{\cong} \ar@{^{(}->}[d]  &  
 {}^{L} H_{\Dol}^{p,q}(M,J) \ar@{^{(}->}[d]  \\
\Hh_\delb^{p,q} \cap \Hh_{\mub}^{p,q}  \ar@{^{(}->}[r] &  \Hh_{\delb_\mub}^{p,q}  \ar@{^{(}->}[r] & 
H_{\Dol}^{p,q}(M,J) 
}
\]
and an isomorphism
\[{}^{L} H_{\Dol}^{p,q}(M,J)\cong {}^{L} \Hh_{\delb_\mub}^{p,q}.\]
In particular, the numbers $\dim {}^{L} \Hh_{\delb_\mub}^{p,q}$ are metric-independent.
\end{enumerate}
\end{cor}

\begin{ex}[Filiform nilmanifold]\label{Liealg}
Consider the nilpotent  real Lie algebra $\mathfrak{g}$ generated by $\{X_1, X_2, X_3, X_4\}$
with the only non-trivial Lie brackets given by 
\[
[X_1 ,X_i  ] = X_{i+1} \quad \textrm{for} \quad i=2,3.
\]
Define an almost complex structure on $\mathfrak{g}$ by letting  $JX_1 = X_2$ and $J X_3  =X_4$.

Consider the complexified Lie algebra $\mathfrak{g}_\CC$, generated by $\{A,\overline{A},B,\overline{B}\}$
with $A=X_1-iJX_1 = X_1 -iX_2$ and $B=X_3-iJX_3 = X_3 - iX_4 $.  The only non-trivial Lie brackets are then given by
\[
[A,B] = [A,\bar B] = [\bar A, B] = [\bar A, \bar B] = \frac{-1}{2i} (B -\bar B),
\]
and 
\[
[A,\bar A ] = i ( B + \bar B).
\]
By dualizing, there is a free commutative differential bigraded algebra 
\[\cA^{*,*}_{\mathfrak{g}}\cong \Lambda(a,b,\overline{a},\overline{b})\]
with generators of bidegrees $|a|=|b|=(1,0)$ and $|\overline{a}|=|\ov{b}|=(0,1)$
and the only non-trivial differentials on generators given by
\[
\mub b = \frac{1}{2i} \bar a \bar b, \quad \quad 
\delb b =\frac{1}{2i}  \left( a \bar b - b \bar a \right) - i a \bar a,   \quad \quad
\del b = \frac{1}{2i} ab,
\]
together with the corresponding complex conjugate equations.

Consider the 4-dimensional nilmanifold $M:= \Gamma \backslash G$, where $G$ has Lie algebra $\mathfrak{g}$.
Then
\[{}^{L}H_{\Dol}^{*,*}(M,J)
\cong 
\arraycolsep=4pt\def\arraystretch{1.4}
 \begin{array}{|c|c|c|c|}
 \hline
\,\,\,0\,\,\,&\CC&\CC \\
  \hline
 \CC^2&\CC^4&\CC^2 \\
   \hline
 \CC&\CC&0\\ 
 \hline
\end{array}\text{ and }
{}^{L}E_2(M,J)\cong
\arraycolsep=4pt\def\arraystretch{1.4}
 \begin{array}{|c|c|c|c|}
 \hline
\,\,\,0\,\,\,&\CC&\CC \\
  \hline
 \CC&\CC^2&\CC \\
   \hline
 \CC&\CC&\,\,\,0\,\,\,\\ 
 \hline
\end{array}
\]

The second page is already the complex cohomology $H^*(M,\CC)$ of $M$ (and in fact
  $E_2(M,J) \cong H^*(M,\CC)$ is guaranteed by Theorem \ref{46maxE2deg} below). But note that the first page ${}^{L}H^{*,*}_{\Dol}(M,J)$ contains more information, which is an invariant of the almost complex structure of $M$.
 Recall that the Fr\"{o}licher spectral sequence of a compact complex surface always degenerates at the first page. This example exhibits that this is not true anymore for non-integrable structures on compact four-dimensional manifolds.
  
We may consider another almost complex structure on $M$, defined by 
\[J'X_1=X_4\text{ and } J'X_2=X_3.\]
In this case, the left-invariant Fr\"{o}licher spectral sequence degenerates at $E_1$, and:
\[{}^{L}H_{\Dol}^{*,*}(M,J')
\cong E_{\infty}^{*,*}(M,J')
\cong 
\arraycolsep=4pt\def\arraystretch{1.4}
 \begin{array}{|c|c|c|c|}
 \hline
\,\,\,0\,\,\,&0&\CC \\
  \hline
 \CC^2&\CC^2&\CC^2 \\
   \hline
 \CC&0&0\\ 
 \hline
\end{array}.\]
Note that the left-invariant Dolbeault cohomology, as well as the bigrading induced on de Rham, allow one to distinguish the non-equivalent almost complex structures $J$ and $J'$. 
A simple spectral sequence argument shows that the tables computed above are the two only possible tables for the left-invariant Dolbeault cohomology of a non-integrable left-invariant almost complex structure on a real Lie algebra of dimension 4. We remark as well that the pair $(M,J')$ gives a first example for which the left-invariant Fr\"{o}licher spectral sequence degenerates at an earlier stage than the full
Fr\"{o}licher spectral sequence. Indeed, by the work of \cite{Stelzigandfriends} the latter degenerates precisely at the $E_2$-stage.

The manifold $M$ does not admit any integrable structure, as pointed out to us by Aleksandar Milivojevi\'{c}.
Indeed, since $b_1$ is even, by Kodaira's 
classification of surfaces it would then be K\"{a}hler, and hence it would be formal.
But every formal nilmanifold is a torus, which has $H^1_{\dR}(T^2,\CC)\cong \CC^4$.
In contrast, the above $E_2$-page tells us that $H^1_{\dR}(M,\CC)\cong \CC^2$.  
\end{ex}

\begin{ex}[Kodaira-Thurston manifold]\label{KTExample}
Consider the $4$-dimensional nilmanifold, defined as the quotient 
\[
KT:=H_\Z \times \Z \backslash H \times \R
\]
where $H$ is the $3$-dimensional Heisenberg Lie group, $H_\Z$ is the integral subgroup, and the action is on the left. 
 Its Lie algebra is spanned by $X,Y,Z,W$ with bracket $[X,Y] = -Z$.
On the dual basis $x,y,z,w$, the only non-zero differential is $dz = xy$. 

Consider the non-integrable $J$ given by 
$J(W) = X$ and $J(Z) = Y$, and let $A := X-iJX = X+iW$, and $B := Y-iJY = Y+iZ$, be a basis for the invariant $(1,0)$-vectors. The non-trivial brackets are 
\[
[A,B]  = [A, \bar B] = [ \bar A, B] = [\bar A, \bar B] = -Z = \frac{-1}{2i} \left( B - \bar B \right).
\]
Letting $a,b$ be dual to $A,B$, it follows that the only non-zero components of $d$ in degree one are 
\[
\del b =   \frac{1}{2i}  \, ab ,\quad 
\delb b = \frac{1}{2i}  \, ( a \bar b - b \bar a ),\quad 
\mub b  =   \frac{1}{2i}  \, \bar a \bar b,
\]
and the conjugate equations. This gives:
\[{}^{L} H_{\Dol}^{*,*}(KT,J)\cong E_\infty^{*,*}(KT,J)\cong 
\arraycolsep=4pt\def\arraystretch{1.4}
 \begin{array}{|c|c|c|c|}
 \hline
\,\,\,0\,\,\,&\CC&\CC \\
  \hline
 \CC^2&\CC^4&\CC^2 \\
   \hline
 \CC&\CC&0\\ 
 \hline
\end{array}.
\]
A simple spectral sequence argument shows that the above is the only possible table for the left-invariant Dolbeault cohomology of a left-invariant non-integrable structure on $KT$.

The Kodaira-Thurston manifold also has an integrable structure $J'$ defined by $J'(X) = Y$ and $J'(Z)=W$. The Dobeault cohomology
is well-known (see for instance \cite{C}), given by:
\[{}^{L} H_{\Dol}^{*,*}(KT,J')\cong H_{\Dol}^{*,*}(KT,J')\cong E_\infty^{*,*}(KT,J')\cong 
\arraycolsep=4pt\def\arraystretch{1.4}
 \begin{array}{|c|c|c|c|}
 \hline
\,\,\,\CC\,\,\,&\CC&\CC \\
  \hline
 \CC^2&\CC^2&\CC^2 \\
   \hline
 \CC&\CC&\CC\\ 
 \hline
\end{array}.
\]
Again, any left-invariant integrable structure will have the same table for left-invariant Dolbeault cohomology as the one given by $J'$.
One may easily define a continuous family of almost complex structures from the integrable $J'$ to the non-integrable $J$ above. 
Observing that
\[H^{1,1}_{\Dol}(KT,J')\cong\CC^2\subsetneq \CC^4\cong {}^{L} H^{1,1}_{\Dol}(KT,J)\subseteq H^{1,1}_{\Dol}(KT,J)\]
we deduce that, in contrast with the integrable case, Dolbeault cohomology is not upper semi-continuous for small deformations.

\end{ex}

\section{Maximally non-integrable and nearly K\"{a}hler $6$-manifolds} \label{maxnonint}

In this section, we define the condition for an almost complex structure to be \emph{maximally non-integrable}, and prove for manifolds of dimension $4$ and $6$ that this condition implies that the Fr\"{o}licher spectral sequence degenerates at the second stage. 
A particularly well-behaved family of maximally non-integrable manifolds of dimension $6$ is that of nearly K\"{a}hler manifolds, which are studied in full detail. We derive new algebraic identities for nearly K\"{a}hler $6$-manifolds and use these to compute their Dolbeault cohomology. 
In particular, we conclude that the Dolbeault cohomology groups contain strictly more information than the de Rham cohomology groups and their $(p,q)$-grading. On the other hand, we see that the groups are rather restricted.

\begin{defn} \label{defnmaxnonint}
An almost complex structure $J$ on a manifold $M$ is \emph{maximally non-integrable} if 
the Nijenhuis tensor $N_J:T_x M  \otimes T_x M \to T_xM$ has maximal rank at all points $x \in M$.
\end{defn}

Recall that $\mub: \cA^{p,q} \to \cA^{p-1, q+2} $ is linear over functions and we let 
\[
\cA^{p,q}_x := \bigwedge^p \left(T^*_x M \ot \CC\right)^{1,0}  \otimes  \bigwedge^q \left(T^*_x M \ot \CC \right)^{0,1}
\]
denote the fiber over $x$. 

\begin{lem}
An almost complex structure is maximally non-integrable if and only if
the map
\[
\mub:\cA^{1,0}_x \longrightarrow \cA^{0,2}_x
\]
has maximal rank for all $x \in M$. 
\end{lem}

\begin{proof} The pointwise rank of $N_J$ is equal to the rank of the dual $(N_J \otimes Id)^*$ of the complexification of $N_J$.
So, by Lemma \ref{lem;Nmub}, this rank equals the pointwise rank of the restriction of $\mub+ \mu$  to degree one. Since $\mub$ and $\mu$ are 
supported in bidegrees $(1,0)$ and $(0,1)$, respectively, the rank is maximal if and only if $\mub:\cA^{1,0}_x \longrightarrow \cA^{0,2}_x$
has maximal rank for all $x \in M$. 
\end{proof}

The following two Lemmas are useful in low dimensions.

\begin{lem} \label{4maxnonint}
On an almost complex $4$-manifold the following are equivalent:
\begin{enumerate}
 \item The almost complex structure is maximally non-integrable.
 \item $\mub|_{1,0}$ is a fiberwise surjection
 and $\mub|_{2,0}$ is a fiberwise injection.
 \item For all $x \in M$, the table for the fiberwise dimensions of $H_{\mub,x}^{*,*}$ is given by:
 \[
\arraycolsep=5pt\def\arraystretch{1.3}
 \begin{array}{|c|c|c|c|c|}
 \hline
  
   0  & 1 &1  \\\hline
  2  &4 & 2 \\\hline
1 & 1 &  0 \\
\hline
\end{array}
\]
\end{enumerate}
In particular, on $4$-manifolds, if $J$ is maximally non-integrable then $\mub$ has locally constant rank.
\end{lem}

\begin{proof} The latter two conditions are clearly equivalent. In dimension $4$, condition $(1)$ holds if and only if $\mub: \cA^{1,0} \to \cA^{0,2}$ is fiberwise surjective, which occurs if and only if $\Hh^{0,2}_{\mub,x}=0$ on every fiber. By Serre duality, Proposition \ref{SerreDuality}, this occurs if and only if $\Hh^{2,0}_{\mub,x}=0$ on every fiber, which is equivalent to $\mub: \cA^{2,0} \to \cA^{1,2} $ being fiberwise injective, so $(1)$ is equivalent to $(2)$ and $(3)$. The last claim follows from condition $(3)$, since the 
rank of $\mub$ restricted to $(p,q)$ can be computed from the fiberwise Hodge decomposition to be equal to one in bidegrees $(1,0)$ and $(2,0)$, and zero otherwise.
\end{proof}

\begin{rmk}
There are nontrivial topological obstructions to admitting a nowhere integrable almost complex structure on a $4$-manifold. Armstrong \cite{Arm} showed that if $(M, J)$ is a $4$-dimensional compact almost complex manifold with the Nijenhuis tensor non-vanishing at each point, then the signature and Euler characteristic of $M$ satisfy $5\chi(M) + 6\sigma(M)= 0$.
\end{rmk}

\begin{lem} \label{6maxnonint}
On an almost complex $6$-manifold the following are equivalent:
\begin{enumerate}
 \item The almost complex structure is maximally non-integrable.
 \item $\mub|_{1,0}$ is a fiberwise isomorphism.
 \item For all $x \in M$, the table for the fiberwise dimensions of $H_\mub^{*,*}$ is given by:
\[\arraycolsep=5pt\def\arraystretch{1.3}
\begin{array}{|c|c|c|c|c|}
 \hline
 0  &0 & 0  &  1\\\hline
 0&  6  & 8 & 3 \\\hline
 3 & 8 & 6& 0 \\\hline
1& 0 & 0 &  0 \\
\hline
\end{array}\]
\end{enumerate}
In particular, on $6$-manifolds, if $J$ is maximally non-integrable then $\mub$ has locally constant rank.
\end{lem}

\begin{proof}
In dimension $6$, maximally non-integrable is equivalent to $\mub: \cA^{1,0} \to \cA^{0,2} $ being a fiberwise isomorphism, since these fibers have the same dimension. So conditions $(1)$ and $(2)$ are equivalent, and clearly $(3)$ implies $(2)$, so it remains to show that $(2)$ implies $(3)$.

Suppose $\mub: \cA^{1,0} \to \cA^{0,2} $ is a fiberwise isomorphism, and let $\mub|_{p,q}$ denote $\mub$ restricted to
$\cA^{p,q}$. Since $\mub|_{1,0}$ is fiberwise surjective, $\Hh^{0,2}_\mub=0$ on every fiber, so by Serre duality,
$\Hh^{3,1}_\mub=0$ on every fiber. It remains to show that $\mub|_{2,0}$ and $\mub|_{3,0}$ are fiberwise injective.
In any choice of frame $(\xi_1 , \xi_2, \xi_3) $ of $ \cA^{1,0}_x$, with conjugate $(\bar \xi_1, \bar \xi_2, \bar \xi_3) \in \cA^{0,1}_x$, we may write
\[
\mub (\xi_i) = \sum_{j<k} \lambda^{jk}_i \, \bar \xi_j \wedge \bar \xi_k,
\]
so that for $i \neq j$
\[
 \mub (\xi_i \wedge \xi_j) =  \sum_{m<n} \lambda^{mn}_i \, \bar \xi_m \wedge \bar \xi_n \wedge \xi_j
 -
 \sum_{r<s} \lambda^{rs}_j \,  \xi_i \wedge  \bar \xi_r \wedge \bar \xi_s.
 \]
 For each fixed $i$ and $j$, the constants $\lambda^{mn}_i$ and $\lambda^{rs}_j$ are not all zero, and the vectors 
 $\bar \xi_m \wedge \bar \xi_n \wedge \xi_j$ and $\xi_i \wedge  \bar \xi_r \wedge \bar \xi_s$ are linearly independent, which shows $\mub^{2,0}$ is injective since the frame was arbitrary. A similar argument applied to 
 \[
 \mub ( \xi_1 \wedge \xi_2 \wedge \xi_3)=
 \sum_{r<s} \lambda^{rs}_1 \,  \bar \xi_r \wedge  \bar \xi_s \wedge \xi_2 \wedge  \xi_3
 -
 \sum_{r<s} \lambda^{rs}_2 \, \xi_1 \wedge \bar \xi_r \wedge  \bar \xi_s \wedge  \xi_3
 +
 \sum_{r<s} \lambda^{rs}_3 \,  \xi_1 \wedge \xi_2 \wedge  \bar \xi_r \wedge \bar \xi_s
  \]
  shows that $\mub^{3,0}$ is injective: the constants $\lambda^{rs}_i$ are not all zero, and the summands are linearly independent.
  
 The last claim follows from condition $(3)$, since the 
rank of $\mub$ restricted to $(p,q)$ can be computed from the fiberwise Hodge decomposition to be equal to $3$ in bidegrees $(1,0)$,$(2,0)$, $(2,1)$, and $(3,0)$, equal to $1$ in bidegrees $(1,1)$ and $(3,0)$, and zero otherwise.
 \end{proof}

\begin{thm} \label{46maxE2deg}
The Fr\"olicher spectral sequence for a maximally non-integrable almost complex structure on a $4$-manifold or a $6$-manifold degenerates at the $E_2$-term.
\end{thm} 

\begin{proof} By the previous two Lemmas, the $E_2$ page of the Fr\"olicher spectral sequence 
is given in dimensions $4$ and $6$ by the following diagrams, since  $H_\mub^{p,q} =0$ implies $E_2^{p,q} =0$:
\[
\arraycolsep=5pt\def\arraystretch{1.3}
 \begin{array}{|c|c|c|c|c|}
 \hline
  
   0  & &  \\\hline
    & &  \\\hline
\,\,\,  &  \,\,\, &  0 \\
\hline
\end{array}
\quad
\quad 
 \begin{array}{|c|c|c|c|c|}
 \hline
  0 & 0 & 0  &  \\\hline
 0 &    &  &  \\\hline
  &  & & 0 \\\hline
\,\,\, \,& 0 & 0 &  0 \\
\hline
\end{array}
\]
It follows that in dimension $4$ or $6$, for $k \geq 2$, the differential $d_k: E_k^{p,q} \to E_k^{p+k,q-k+1}$ vanishes since, for all $p,q$, either $E_k^{p,q} =0$ or $E_k^{p+k,q-k+1} =0$.
\end{proof}

\begin{rmk}
The above result is optimal, in the sense that there exist
maximally non-integrable compact
$4$- and $6$-manifolds whose left-invariant spectral sequence does not degenerate at $E_1$, as shown in Examples \ref{Liealg} and \ref{su2su2} respectively. In addition, by \cite{Stelzigandfriends} the non-left invariant spectral sequence never degenerates at $E_1$ in the case of 4- and 6-dimensional maximally non-integrable manifolds.
\end{rmk}

We can compute abstractly the Dolbeault cohomologies as follows.

\begin{ex}\label{HDolMNI4} The Dolbeault cohomology of a maximally non-integrable almost complex $4$-manifold is given by 
\[
H^{*,*}_{\Dol}\cong
\arraycolsep=6pt\def\arraystretch{2}
 \begin{array}{|c|c|c|c|c|}
 \hline
 0&{\cA^{1,2}/ \{\mub \alpha+\delb \beta\}}&\Cok(\delb)  \\\hline
 \Cok(\delb)& \{\omega; \delb \omega =\mub \eta \}/\{ \delb \beta; \mub \beta=0\} &\Ke(\delb)    \\\hline
 \Ke(\delb)&\Ke(\delb)\cap\Ke(\mub)&0 \\\hline
\end{array}.\]
Note that for all $(p,q)\neq (1,1)$, we have isomorphisms $H^{p,q}_{\Dol}(M)\cong \Hh_{\delb_\mub}^{p,q}$.
Indeed, the bottom and top rows follow from Lemma \ref{bottom_top} and in bidegree $(0,1)$ it follows from Lemma \ref{H01}.
In bidegree $(2,1)$,  $H_{\Dol}^{2,1}\cong \Ke(\delb)\cong \Hh_{\del_\mub}^{0,1}$ by Lemma \ref{4maxnonint}, making $\mub|_{2,0}$ injective and therefore $\mub^*|_{1,2}$ surjective.
\end{ex}

\begin{ex}\label{HDolMNI6}
Similarly, for a maximally non-integrable $6$-manifold we have:
\begin{equation*}
\resizebox{1\hsize}{!}{$
H^{*,*}_{\Dol}\cong
\arraycolsep=6pt\def\arraystretch{2}
 \begin{array}{|c|c|c|c|c|}
 \hline
0&0&0&\Cok(\delb)  \\\hline
0& \cA^{1,2}/\{\mub \alpha+\delb \beta; \mub \beta=0\}&  \cA^{2,2}/\{\mub \alpha+\delb \beta; \mub \beta=0\} &\Ke(\delb)    \\\hline
 \Cok(\delb) & \{\omega\in \Ke(\mub); \delb \omega=\mub \eta\}&\{\omega\in \Ke(\mub); \delb \omega=\mub \eta\}&0 \\\hline
 \Ke(\delb)&0&0&0 \\\hline
\end{array}$}
\end{equation*}
In this case, we have isomorphisms $H^{p,q}_{\Dol}(M)\cong \Hh_{\delb_\mub}^{p,q}$ for all $(p,q)\neq (1,2),(2,2)$, 
while in the remaining bidegrees we have $\Hh_{\delb_\mub}^{*,*}=\Ke(\delb_\mub^*)\hookrightarrow H_{\Dol}^{*,*}\cong \Hh_{\mub}^{*,*}/\Img(\delb_\mub)$.
\end{ex}

\subsection{Nearly K\"{a}hler $6$-manifolds}\label{NKSection}
An almost Hermitian manifold $\Mh$ with almost complex structure $J$ is said to be \textit{(strictly) nearly-K\"{a}hler} if $(\nabla_X J ) Y$ is skew-symmetric for all $X$ and $Y$ (and $(\nabla_X J ) Y$ non-zero for some $X,Y$). Equivalently, the covariant derivative $\nabla \omega$ of the fundamental $(1,1)$-form $\omega$ is 
totally skew-symmetric.

In \cite{Verbitsky}, Verbitsky developed a  Hodge theory for nearly K\"{a}hler manifolds by proving
a set of equations analogous to the K\"{a}hler identities.
From these identities, Verbitsky deduced that, in the compact case, a complex form is $d$-harmonic if and only if it
is harmonic with respect to each component of the differential, i.e.,
\[\Ke(\Delta_d)=\Ke(\Delta_\mub)\cap \Ke(\Delta_\delb)\cap\Ke(\Delta_\del)\cap\Ke(\Delta_\mu).\]
This gives a decomposition on complex $d$-harmonic forms
\[\Hh^n_d:=\Ke(\Delta_d)\cap\cA^n=\bigoplus \Hh_d^{p,q}\text{ where }\Hh^{p,q}_{d}:=\Ke(\Delta_d)\cap \cA^{p,q}.\]
Furthermore, he proved that these bigraded spaces are non-trivial only when 
$p,q\in\{1,2\}$ (see Theorem 6.2 of \cite{Verbitsky})
and that 
\[\Hh^{1,2}_{d}\cong \Hh^{2,1}_{d}\text{ and }
 \Hh^{1,1}_{d}\cong \Hh^{2,2}_{d}.\]
Lastly, note that the above bigrading induces a 
Hodge-type decomposition on the complex de Rham cohomology of every 
compact nearly K\"{a}hler 6-manifold.

In this section, we give a detailed study of the Fr\"olicher spectral sequence for a nearly K\"ahler $6$-manifold.
We first prove degeneration at the second page.

\begin{lem} \label{NK6maxnonint}
A nearly K\"ahler $6$-manifold is maximally non-integrable. 
\end{lem}

\begin{proof}
According to \cite{Verbitsky2}, or  \cite{Verbitsky}, there is a local frame $(\xi_1, \xi_2, \xi_3) \in \cA^{1,0}$, and conjugate $(\bar \xi_1, \bar \xi_2, \bar \xi_3) \in \cA^{0,1}$ such that
 \begin{equation} \label{eq;mubptwise}
 \mub (\xi_1) =  \lambda \, \bar \xi_ 2 \wedge \bar \xi_3 \quad 
 \mub (\xi_2) = - \lambda \, \bar \xi_ 1\wedge \bar \xi_3  \quad 
 \mub (\xi_3)  =  \lambda \, \bar \xi_ 1 \wedge \bar \xi_2
\end{equation} 
for some $\lambda \neq 0$.
Equation (\ref{eq;mubptwise}) shows 
$\mub_{(1,0)}: \cA^{1,0} \to \cA^{0,2}$ is a fiberwise isomorphism, so the claims follows from Lemma \ref{6maxnonint}.
\end{proof}

From the previous lemma, Theorem \ref{46maxE2deg}, and Corollary \ref{HdRpq}, it follows that
\begin{cor} \label{NK6E2}
The Fr\"olicher spectral sequence for a nearly K\"ahler $6$-manifold degenerates at the $E_2$-term,
and for all $p,q$
\[H_{\dR}^{p,q}:=E_2^{p,q}\cong \Hh_d^{p,q}.\]
In particular, for compact nearly K\"ahler $6$-manifolds, the Hodge-type decomposition induced in cohomology
agrees with the metric-independent bigrading induced by the almost complex structure.
\end{cor}

\begin{rmk}
It is an interesting geometric question whether a compact almost complex $6$-manifold admits a nearly K\"ahler metric. Verbitsky showed there is at most one such metric (up to a constant) for each almost complex structure, \cite{Verbitsky2}. 
Since the degeneration depends only on the almost complex structure, and not the metric, it follows that that if the Fr\"olicher spectral sequence for an almost complex $6$-manifold does not degenerate at $E_2$, then there is no metric for which the structure is nearly K\"ahler. 
\end{rmk}

In what follows, we give a careful study of the Dolbeault cohomology for nearly K\"{a}hler manifolds. To do so, we first review some known background on the algebra of operators on the differential forms of a  nearly  K\"ahler $6$-manifold, as well as establish some new algebraic identities which may be of interest in their own right. These will be used to facilitate the calculations on the $E_1$-page.

Let $[A,B]:=AB-(-1)^{|A||B|}BA$ denote the graded commutator of operators $A$ and $B$. 
The following commutation relations are verified:

\begin{lem}[\cite{Verbitsky}, Proposition 4.3]\label{commutators}
For a nearly K\"{a}hler $6$-manifold, we have:
\begin{align*}
[\mu^*,\delb ]=[\mub^*,\del ]=[\mu,\delb^* ]=[\mub,\del^* ]=0,\\
[\delb^*,\del]=-[\mu,\del^*]=-[\mub^*,\delb],\\
[\del^*,\delb]=-[\mub,\delb^*]=-[\mu^*,\del],\\
[\mu,\mub^*]= [\mub,\mu^*]=0.
\end{align*}
\end{lem}

Given an almost Hermitian manifold $\Mh$ we will denote by $\omega$ the associated fundamental $(1,1)$-form and by $L$ the operator given by $L(\eta):=\omega\wedge \eta$. Let $\Lambda$ be the adjoint to $L$. We have the following set of nearly K\"{a}hler identities:

\begin{lem}[\cite{Verbitsky}, Theorem 3.1 and Proposition 5.1]\label{identities}
For a nearly K\"{a}hler $6$-manifold, we have:
\begin{enumerate}
 \item $[L,\del^*]=i\delb$,\,\,\, $[\Lambda,\del]=i\delb^*$,\,\,\, $[L,\delb^*]=-i\del$\,\,\, and $[\Lambda,\delb]=-i\del^*$.
 
 \item $[L,\mu^*]=2i\mub$, $[\Lambda,\mu]=2i\mub^*$, $[L,\mub^*]=-2i\mu$ and $[\Lambda,\mub]=-2i\mu^*$.
 
\item If $\eta \in \cA^{p,q}$ then $(\del\delb+\delb\del) \eta =-i\lambda^2(p-q)L \eta $.
\end{enumerate}
\end{lem}

We now deduce the following:
\begin{lem}\label{equalDeltas}
For any compact nearly K\"{a}hler $6$-manifold, we have identities
\[
\Delta_\delb+2\Delta_\mu=\Delta_\del+2\Delta_\mub.
\]
In particular, for all $p,q$ we have
\[
\Hh_\delb^{p,q}\cap \Hh_\mu^{p,q}=\Hh_\del^{p,q}\cap\Hh_\mub^{p,q}.
\]
Moreover, restricted to bidegrees $p=q$, or $p+q=3$, we have 
\[\Delta_\mub=\Delta_\mu\text{ and } \, \Delta_\delb=\Delta_\del.\]
\end{lem}
\begin{proof}
Using (2) of Lemma \ref{identities} we may write $\Delta_\mu$ as
\begin{align*}
\Delta_\mu=\mu\mu^*+\mu^*\mu={i\over 2}\left(\mu[\Lambda,\mub]+[\Lambda,\mub]\mu\right)
={i\over 2}\left(\mu \Lambda\mub -\mu\mub \Lambda+\Lambda\mub \mu -\mub \Lambda \mu\right).
\end{align*}
Likewise, for $\Delta_\mub$ we have:
\begin{align*}
\Delta_\mub=\mub\mub^*+\mub^*\mub=-{i\over 2}\left(\mub[\Lambda,\mu]+[\Lambda,\mu]\mub\right)
=-{i\over 2}\left(\mub \Lambda\mu -\mub\mu \Lambda+\Lambda\mu \mub -\mu \Lambda \mub\right)
\end{align*}
All together this gives 
\begin{align*}
 \Delta_\mu-\Delta_\mub={i\over 2}\left(\Lambda(\mu\mub+\mub\mu) - (\mu\mub+\mub\mu)\Lambda\right)
 =-{i\over 2}\left(\Lambda(\del\delb+\delb\del) - (\del\delb+\delb\del)\Lambda\right).
\end{align*}
Let $\eta \in \cA^{p,q}$.
Then by (3) of Lemma \ref{identities} we have 
\[(\Delta_\mu-\Delta_\mub)\eta=-{\lambda^2\over 2}(p-q)[\Lambda,L]\eta.\]
It now suffices to use the fact that 
$[\Lambda,L]\eta=(p+q-3)\eta$ (see for example \cite{GrHa}), to obtain
\[(\Delta_\mu-\Delta_\mub)\eta={\lambda^2\over 2}(3-p-q)(p-q)\eta.\]
In Corollary 3.3 in \cite{Verbitsky}, Verbitsky shows that 
\[(\Delta_\del-\Delta_\delb)\eta=\lambda^2(3-p-q)(p-q)\eta.\]
Combining these equations, we obtain the Lemma.
\end{proof}

By the previous Lemma and the conjugation isomorphism we immediately have:

\begin{cor} \label{p+q=3sym}
For a nearly K\"{a}hler $6$-manifold, and all $p+q=3$, we have
\[
\dim \Hh_\delb^{p,q} =\dim \Hh_\delb^{q,p}
\text{ and }
\dim \left( \Hh_{\delb}^{p,q}\cap \Hh_\mub^{p,q} \right)
= 
\dim \left(\Hh_{\delb}^{q,p}\cap \Hh_\mub^{q,p} \right).
\]
\end{cor}

By Lemma \ref{NK6maxnonint} and the definition of maximally non-integrable, $\Hh^{p,q}_\mub=0$ for 
$\binom 3 p \binom 3 q \leq \binom{3}{p-1} \binom{3}{q+2}$. The Serre-dual groups vanish as well. Then Lemma \ref{equalDeltas}   implies that $\Hh_\mu^{p,q}=\Hh_\mub^{p,q}$ for all pairs $(p,q)$ such that 
$p+q\neq 1$ and $p+q\neq 5$. 
So, $\mub$-harmonic forms and $\mu$-harmonic forms of type $(p,q)$ 
coincide in the range depicted in grey in the table below:
\[
\arraycolsep=6pt\def\arraystretch{1.4}
 \begin{array}{|c|c|c|c|c|}
 \hline
 \greycell 0 & \greycell 0 &   &  \greycell \\
\hline
 0 \greycell &   \greycell  & \greycell  &  \\
\hline
  &\greycell  & \greycell & \greycell 0  \\
\hline
 \greycell  & &\greycell  0 & \greycell 0 \\
\hline
\end{array}
\]

An easy computation shows that for the remaining cases we have 
\begin{align*}
 \Hh_\mu^{0,1} &=0 & \Hh_\mu^{1,0}  &=\cA^{1,0} & \Hh_\mu^{3,2} &=0 &\Hh_\mu^{2,3}& =\cA^{2,3} \\
 \Hh_\mub^{0,1} &=\cA^{0,1} & \Hh_\mub^{1,0} &=0 & \Hh_\mub^{3,2}  &=\cA^{3,2}& \Hh_\mub^{2,3} &=0.
\end{align*}

Lemma \ref{6maxnonint} guarantees that, for nearly K\"ahler $6$-manifolds,  $\mub$ has locally constant rank. Therefore, we may consider the spaces $\Hh_{\delb_\mub}^{p,q}$ defined in subsection \ref{LCR}. We use the established harmonic theory to compute these.

\begin{lem}\label{HDol_delb_mub_kah}
Let $M$ be a compact nearly K\"{a}hler $6$-manifold. For the $(p,q)$-range depicted in grey in the table below we have
\[\Hh^{p,q}_{d}=\Hh_\delb^{p,q}\cap\Hh_\mub^{p,q}=\Hh_{\delb_\mub}^{p,q}.\]
\[
\arraycolsep=6pt\def\arraystretch{1.4}
 \begin{array}{|c|c|c|c|c|}
 \hline
 \greycell 0 & \greycell 0 & \greycell  0 &  \greycell \\
\hline
 0 \greycell &   \greycell  &   & \\
\hline
  &  & \greycell & \greycell 0  \\
\hline
 \greycell  & \greycell0 &\greycell 0 & \greycell 0 \\
\hline
\end{array}
\]

\end{lem}
\begin{proof}
The first identity follows from Lemma \ref{equalDeltas} together with the above vanishing results of $\Hh_\mub$ and $\Hh_\mu$. By Lemma \ref{HIT} we have
$\Hh_\delb^{p,q}\cap\Hh_\mub^{p,q}\subseteq\Hh_{\delb_\mub}^{p,q},$
so let us prove the converse inclusion. The only non-trivial cases are $(p,q)=(1,2)$ and $(p,q)=(2,1)$.
Throughout this proof we will systematically use the commutation relations of Lemma \ref{commutators}.

Let $\omega\in \Ke(\Delta_{\delb_\mub})\cap\Hh_{\mub}^{1,2}$.
Then we have 
\[\delb^* \omega=\Hh_\mu(\delb^*\omega)+G_\mu \mu^*\mu \delb^*\omega+G_\mu \mu\mu^*\delb^*\omega.\]
For degree reasons, the last summand in the above equation is trivial.
Since $\{\delb^*,\mu\}=0$ we have 
\[G_\mu \mu^*\mu \delb^*\omega=-G_\mu \mu^*\delb^*\mu \omega.\]
Since $\omega\in \Hh^{1,2}_{\mub}=\Hh^{1,2}_{\mu}$ this term is also trivial.
Therefore we have \[\delb^*\omega=\Hh_\mu(\delb^*\omega)=\Hh_\mub(\delb^*\omega)=0.\]
Using the $\mub$-decomposition for $\del^*\omega$ we have:
\[\del^*\omega=\Hh_\mub(\del^* \omega)+G_\mub\mub \mub^*  \del^* \omega+G_\mub\mub^* \mub  \del^* \omega\]
where, for degree reasons, the last summand is trivial. We may write the second summand as
\[G_\mub\mub \mub^*  \del^* \omega=-G_\mub \mub (\del^*\mub^*+\delb^*\delb^*)\omega=0.\]
Also, note that by degree reasons, we have $\Hh_\mub(\del^*\omega)=0$. This gives $\del^*\omega=0$.
Lastly, using $\{\mu,\del^*\}=\{\mub^*,\delb\}$,
\[\mub^*\delb \omega=(\del^* \mu-\delb\mub^*)\omega=0.\]
This proves that $\delb \omega=0$, since $\Hh_\mub^{1,3} = 0$, so that $\omega\in\Ke(\Delta_\delb)$.
The proof for the case $\omega\in \Ke(\Delta_{\delb_\mub})\cap\Hh_{\mub}^{2,1}$ follows dually.
\end{proof}

\begin{rmk} ($\Delta_{\delb_\mub}$ is not elliptic) \label{delbmubnonelliptic}
Notice that for $p=1$ and $q=0,1,2$ the spaces $\Hh_\mub^{p,q}$ are vector bundles with fiber rank $0,8,$ and $6$, respectively. The operator $\Delta_{\delb_\mub}:\Hh_\mub^{1,1}  \to \Hh_\mub^{1,1}$ has symbol map factoring through the bundle $\Hh_\mub^{1,0} \oplus \Hh_\mub^{1,2}$, which has strictly lower fiberwise rank than $\Hh_\mub^{1,1}$. This shows the symbol is not an isomorphism.
\end{rmk}

Putting together the above results, together with Example \ref{HDolMNI6}, the Dolbeault cohomology of a compact nearly K\"{a}hler 6-manifold is expressed in terms of harmonic forms in the following way:
\[
H^{*,*}_{\Dol}\cong
\arraycolsep=6pt\def\arraystretch{2}
 \begin{array}{|c|c|c|c|c|}
 \hline
0&0&0&H^{3,3}_{\dR}  \\\hline
0& \Hh_\mub^{1,2}/\Img(\delb_\mub)&  \Hh_\mub^{2,2}/\Img(\delb_\mub)&\Hh_{\delb_\mub}^{3,2}   \\\hline
 \Hh_{\delb_\mub}^{0,1} &\Hh_{\delb_\mub}^{1,1}&\Hh^{2,1}_{d}&0 \\\hline
 \Hh^{0,0}_{d}&0&0&0 \\\hline
\end{array}\]
where the differential $\delta_1^{0,1}$ is injective and $\delta_1^{2,2}$ is surjective. 

\begin{cor} \label{NK6H12}
If a compact almost complex $6$-manifold has $H^{2,1}_{\Dol} \neq H^{2,1}_{\dR}$, then there is no metric for which it is nearly K\"ahler.
\end{cor}

Finally, we note that the groups $H^{*,*}_{\Dol}$ for bidegrees $(0,1)$,  $(1,1)$, $(2,2)$, and $(3,2)$ are in general not equal to $H^{*,*}_{\dR}$.
In fact,  Example \ref{su2su2} is a nearly K\"{a}hler structure on $SU(2)\times SU(2)$
with
\[\CC^3\subseteq H_{\Dol}^{*,*}\neq H^{*,*}_{\dR}=0,\] 
for all bidegrees $(0,1)$, $(1,1)$, $(2,2)$, and $(3,2)$.
 This shows that  Dolbeault cohomology contains strictly more information than de Rham cohomology in the nearly K\"{a}hler case.

 \appendix
 
 \section{Higher stages of the Fr\"{o}licher spectral sequence}
 
We include a description of all stages of the Fr\"{o}licher spectral sequence of an almost complex manifold $M$.

\begin{thm}\label{description_ss}
For $r\geq 1$, the term $(E_r^{*,*}(M),\delta_r)$ may be described by a quotient
\[E_r^{p,q}(M)\cong Z_r^{p,q}(M)/B_r^{p,q}(M)\]
where:
\begin{enumerate}
 \item For $r=1$ we have:
 \[Z_1^{p,q}(M)\cong \left\{\omega\in \cA^{p,q}\cap\Ke(\mub); \exists\, \omega_1 \text{ with }\delb \omega=\mub \omega_1\right\}
\]
and 
\[B_1^{p,q}(M)\cong
\left\{\eta\in \cA^{p,q}; \exists\, \eta_1, \eta_2\text{ with }\eta=\mub \eta_1+\delb \eta_2\text{ and } \mub \eta_2=0\right\}.
\]
The differential $\delta_1:E_1^{p,q}(M)\to E_1^{p+1,q}(M)$  is given by
\[\delta_1[\omega]=[\del \omega-\delb \omega_1].\]
\item For $r=2$ we have:
\[Z_2^{p,q}(M)\cong \left\{\omega\in \cA^{p,q}\cap\Ke(\mub); \exists\, \omega_1,\omega_2\text{ with }
\delb \omega=\mub \omega_1, \del \omega =\mub \omega_2+\delb \omega_1\right\}
\]
and 
\[B_2^{p,q}(M)\cong 
\left\{\eta\in \cA^{p,q}; \exists\, \eta_1, \eta_2, \eta_3 \text{ with }
\begin{array}{l}
\eta=\mub \eta_1+\delb \eta_2+\del \eta_3,\\
0=\mub \eta_2+\delb \eta_3,\\
0=\mub \eta_3
\end{array}
\right\}.
\]
The differential $\delta_2:E_2^{p,q}(M)\to E_2^{p+2,q-1}(M)$ is given by \[\delta_2[\omega]:=[\mu \omega-\del \omega_1 -\delb \omega_2 ].\]

\item For $r=3$ we have:
\[Z_3^{p,q}(M)\cong \left\{\omega\in \cA^{p,q}\cap\Ke(\mub); \exists\, \omega_1,\omega_2,\omega_3\text{ with }
\begin{array}{l}
\delb \omega=\mub \omega_1,\\
\del \omega =\mub \omega_2+\delb \omega_1,\\
\mu \omega=\mub \omega_3+\delb \omega_2+\del \omega_1
\end{array}
\right\}
\]
and 
\[B_3^{p,q}(M)\cong
\left\{\eta\in \cA^{p,q}; \exists\, \eta_1,\eta_2,\eta_3,\eta_4\text{ with }
\begin{array}{l}
\eta=\mub \eta_1+\delb \eta_2+\del \eta_3+\mu \eta_4\\
0=\mub \eta_2+\delb \eta_3+\del \eta_4,\\
0=\mub \eta_3+\delb \eta_4,\\
0=\mub \eta_4\\
\end{array}
 \right\}
\]
The differential $\delta_3:E_3^{p,q}(M)\to E_3^{p+3,q-2}(M)$ is given by \[\delta_3[\omega]:=[-\mu \omega_1-\del \omega_2-\delb \omega_3].\]

\item For $r\geq 4$ we have:
\[Z_r^{p,q}(M)\cong 
\left\{
\begin{array}{l}
\omega\in \cA^{p,q}\cap\Ke(\mub);\\
\exists\, \omega_1,\cdots,\omega_r
\end{array}
\text{ with }
\begin{array}{l}
\delb \omega=\mub \omega_1,\\
\del \omega =\mub \omega_2+\delb \omega_1,\\
\mu \omega=\mub \omega_3+\delb \omega_2+\del \omega_1\\
\text{ and }\\
0=\mub \omega_i+\delb \omega_{i-1}+\del \omega_{i-2}+\mu \omega_{i-3} \\
\text{ for all }4\leq i\leq r,
\end{array}
\right\}
\]

\[B_r^{p,q}(M)\cong 
\left\{
\begin{array}{l}
 \eta\in \cA^{p,q};\\
  \exists\, \eta_1,\cdots,\eta_{r+1}
\end{array}
\text{ with }
\begin{array}{l}
\eta=\mub \eta_1+\delb \eta_2+\del \eta_3+\mu \eta_4,\\
0=\mub \eta_i+\delb \eta_{i+1}+\del \eta_{i+2}+\mu \eta_{i+3}\\
\text{ for all } 2\leq i\leq r-2,\\
0=\mub \eta_{r-1}+\delb \eta_r+\del \eta_{r+1},\\
0=\mub \eta_r+\delb \eta_{r+1},\\
0=\mub \eta_{r+1}
\end{array}
 \right\}
\]
The differential $\delta_r:E_r^{p,q}(M)\to E_r^{p+r,q-r+1}(M)$ is given by 
\[\delta_r[\omega]:=[-\mu \omega_{r-2}-\del \omega_{r-1}-\delb \omega_r].\]
\end{enumerate}
In all of the above spaces, the elements $\omega_i$ and $\eta_i$ appearing in the formulas are assumed
to be of the pure bidegree prescribed by the equations: $\omega_i\in \cA^{p+i,q-i}$ and $\eta_i\in \cA^{p+2-i,q-3+i}$.
\end{thm}
\begin{proof}
By Lemma \ref{comparison_ss} it suffices to describe the terms 
$E_{r}^{*,*}(\cA,\widetilde F)$ for all $r\geq 2$, where $\widetilde{F}$ is the shifted Hodge filtration.
As explained in Remark \ref{multicomplex}, $\widetilde F$ is the column filtration on the total complex of a
multicomplex $(\cA,d_0,d_1,d_2,d_3)$ with four differentials $d_i$ where, due to the shifts in the indices, 
$d_0=\mub$ has bidegree $(0,1)$, $d_1=\delb$ has bidegree $(1,0)$, $d_2=\del$ has bidegree $(2,-1)$ and $d_3=\mu$ has bidegree $(3,-2)$.
The spectral sequence of a general multicomplex has been described by Livernet-Whitehouse-Ziegenhagen in \cite{LW} and gives the above equations in the case of a multicomplex with only four components.
\end{proof}

\begin{rmk}
In the integrable case, the above coincides with the description of the Fr\"{o}licher spectral sequence given in \cite{CFUG}, which is just the spectral sequence of a bicomplex.
\end{rmk}

\bibliographystyle{alpha}

\bibliography{biblio}

\end{document}